\documentclass[a4paper,12pt, abstract=true]{scrartcl}

\usepackage{german}
\usepackage[ansinew]{inputenc}
\usepackage[german,english,french,USenglish]{babel}
\usepackage{latexsym,amssymb,amsfonts,amsmath,bbm}%amsmath
\usepackage{theorem}
\usepackage{xcolor}
\usepackage{graphicx}
\usepackage{esint}
\usepackage{stmaryrd}
\usepackage[normalem]{ulem} % fuer \sout u.s.w.
\usepackage{exscale}
\usepackage{amsfonts}

\definecolor{DarkGreen}{rgb}{0.2,0.6,0.2}

\newtheorem{dummytheorem}{Dummy-Theorem}[section]
\newcommand{\proofendsign}{$\Box$} % \rule{2mm}{2mm}
\newtheorem{lemma}[dummytheorem]{Lemma}
\newtheorem{theorem}[dummytheorem]{Theorem}
\newtheorem{proposition}[dummytheorem]{Proposition}

\newtheorem{corollary}[dummytheorem]{Corollary}

\newtheorem{definition}[dummytheorem]{Definition}

\newenvironment{proof}{{\noindent \bf Proof }}
 {{\hspace*{\fill}\proofendsign\par\bigskip}}

\theorembodyfont{\normalfont}
\newtheorem{remarknorm}[dummytheorem]{Remark}
\newtheorem{examplenorm}[dummytheorem]{Example}

\newcommand{\N}{\mathbb{N}}

\newcommand{\Q}{\mathbb{Q}}
\newcommand{\R}{\mathbb{R}}

\newcommand{\pr}{\mathbb{P}}
\newcommand{\ex}{\mathbb{E}}

\newcommand{\eins}{\mathbbm{1}}

\newcommand{\avatr}{{\rm AV@R}}

\def\eps{\varepsilon}

\newcommand{\OFP}{(\Omega,{\cal F},\mathbb{P})}

\newcommand{\argmax}{\operatornamewithlimits{arg\,max}}

\def\cR{\mathcal R}
\def\cF{\mathcal F}

\def\cL{\mathcal L}
\def\cM{\mathcal M}

\newcommand{\bR}{\mathbb{R}}
\newcommand{\bP}{\mathbb{P}}
\newcommand{\bE}{\mathbb{E}}
\newcommand{\bN}{\mathbb{N}}
\def\ua{\uparrow}

\def\bcswitch{\left\{\renewcommand{\arraystretch}{1.2}\begin{array}{c@{~:~}c}}

\def\ecswitch{\end{array}\right.}

\allowdisplaybreaks

\begin{document}

%%%%%%%%%%%%%%%%%%%%%%%%%%%%%%%%%%%%%%%%%%%%%%%%%%%%%%%%%%%%%%%%
%%%%%%%%%%%%%%%%%%%%%%%%%%%%%%%%%%%%%%%%%%%%%%%%%%%%%%%%%%%%%%%%
%%%%%%%%%%%%%%%%%%%%%%%%%%%%%%%%%%%%%%%%%%%%%%%%%%%%%%%%%%%%%%%%

\title{Domains of weak continuity of statistical functionals with a view toward robust statistics}

\author{
Volker Krätschmer\footnote{Faculty of Mathematics, University of Duisburg--Essen, {\tt volker.kraetschmer@uni-due.de}}
\and {\setcounter{footnote}{3}}Alexander Schied\footnote{Department of Mathematics, University of Mannheim, {\tt schied@uni-mannheim.de}} {\setcounter{footnote}{6}}
\and Henryk Z\"ahle\footnote{Department of Mathematics, Saarland University, {\tt zaehle@math.uni-sb.de}\hfill\break
The second author gratefully acknowledges support by Deutsche Forschungsgemeinschaft through the Research Training Group RTG 1953. The third author gratefully acknowledges support by BMBF through the project HYPERMATH under grant 05M13TSC. }}

\date{}

\maketitle

\begin{abstract}
Many standard estimators such as several maximum likelihood estimators or the empirical estimator for any law-invariant convex risk measure are not (qualitatively) robust in the classical sense. However, these estimators may nevertheless satisfy a weak \cite{Kraetschmeretal2012,Kraetschmeretal2014} or a local \cite{Zaehle2016} robustness property on relevant sets of distributions. One aim of our paper is to identify  sets of local robustness, and to explain the benefit of the knowledge of such sets. For instance, we will be able to demonstrate that many maximum likelihood estimators are robust on their natural parametric domains.  A second aim consists in
extending the general theory of robust estimation to our local framework.
 In particular we provide a corresponding Hampel-type theorem linking local robustness of a plug-in estimator with a certain continuity condition.
\end{abstract}


%%%%%%%%%%%%%%%%%%%%%%%%%%%%%%%%%%%%%%%%%%%%%%%%%%%%%%%%%%%%%%%%
%%%%%%%%%%%%%%%%%%%%%%%%%%%%%%%%%%%%%%%%%%%%%%%%%%%%%%%%%%%%%%%%
%%%%%%%%%%%%%%%%%%%%%%%%%%%%%%%%%%%%%%%%%%%%%%%%%%%%%%%%%%%%%%%%
%%%%%%%%%%%%%%%%%%%%%%%%%%%%%%%%%%%%%%%%%%%%%%%%%%%%%%%%%%%%%%%%
%%%%%%%%%%%%%%%%%%%%%%%%%%%%%%%%%%%%%%%%%%%%%%%%%%%%%%%%%%%%%%%%
%%%%%%%%%%%%%%%%%%%%%%%%%%%%%%%%%%%%%%%%%%%%%%%%%%%%%%%%%%%%%%%%

\newpage

\section{Introduction and problem statement}\label{Introduction}

Recently, in \cite{Zaehle2016} qualitative robustness of plug-in estimators was considered as a local property, i.e., on strict subsets of the natural domain of the corresponding statistical functional, and a respective Hampel-type criterion was proven. The latter says that if the statistical functional is continuous for a certain topology finer than the weak topology, then qualitative robustness holds on every set of distributions on which the relative weak topology coincides with the finer topology. Such sets of distributions were characterized in \cite{Zaehle2016}, but the provided characterization is rather technical and not at all useful for checking the concurrence of the topologies for any given set. The aim of the present paper is to provide more useful characterizations of such sets, and to illustrate their use in the context of qualitative robustness. Compared to \cite{Zaehle2016} we will also allow for more general topologies on sets of distributions which will turn out to increase the flexibility to check qualitative robustness for statistical functionals. As applications, robustness of maximum likelihood estimators and of empirical estimators of law-invariant convex risk measures are studied in detail. In particular we will demonstrate that many maximum likelihood estimators are robust on their natural parametric domains and even on broader sets. A further field of application is quantitative risk management. In recent contributions in this field the property of robustness has been pointed out as an important requirement for risk assessment; see, for instance, \cite{Contetal2010,Embrechtsetal2014,Kraetschmeretal2014}. Again the empirical estimators of well-founded statistical functionals like those associated with law-invariant convex risk measures fail to be robust but might satisfy this property on domains of interest.

To explain our intension more precisely, let $E$ be a Polish space and ${\cal M}_1$ be the set of all Borel probability measures on $E$. Consider the statistical model
\begin{equation}\label{Statistical Model - 10}
    (\Omega,{\cal F},\{\pr^\theta:\theta\in\Theta\}):=(E^\N,{\cal B}(E)^{\otimes\N},\{\pr^\mu:\mu\in{\cal M}\}),
\end{equation}
where ${\cal M}\subseteq{\cal M}_1$ is any set of Borel probability measures on $E$ and
\begin{equation}\label{Statistical Model - 20}
    \pr^\mu:=\mu^{\otimes\N}
\end{equation}
is the infinite product measure of $\mu$. Note that the coordinate projections on $E^\N$ are i.i.d.\ with law $\mu$ under $\pr^\mu$. For every $\boldsymbol{x}=(x_1,x_2,\ldots)\in E^\N$ and $n\in\N$, we define the empirical probability measure
$$
    \widehat m_n(\boldsymbol{x}):=\widehat m_n(x_1,\ldots,x_n):=\frac{1}{n}\sum_{i=1}^n\delta_{x_i}.
$$
Assume that ${\cal M}$ contains the set
$$
    \mathfrak{E}:=\{\widehat m_n(x_1,\ldots,x_n):\,x_1,\ldots,x_n\in E,\, n\in\N\}\\
$$
of all empirical probability measures. Let $(\Sigma,d_\Sigma)$ be a complete and separable metric space and $T:{\cal M}\to\Sigma$ be any map (statistical functional). The empirical probability measure $\widehat m_n$ induces a nonparametric estimator $\widehat T_n:\Omega\rightarrow\Sigma$ for $T(\mu)$ in the statistical model (\ref{Statistical Model - 10}) through
\begin{equation}\label{Def Plug-in Estimator}
    \widehat T_n(\boldsymbol{x}):=T(\widehat m_n(x)),\quad \boldsymbol{x}=(x_1,x_2,\ldots)\in\Omega,
\end{equation}
provided $\widehat T_n$ is $({\cal F},{\cal B}(\Sigma))$-measurable.

The following Definition \ref{definition of qualitative robustness} generalizes Hampel's classical notion of (qualitative) robustness for the sequence $(\widehat T_n)$ as introduced in \cite{Hampel1971}. Recall from Theorem 2.14 in \cite{Huber1981} that the set of all Borel probability measures on $\Sigma$ equipped with the weak topology is Polish and can be metrized by the Prohorov metric $\pi$. Moreover denote by ${\cal O}_{\rm w}$ the weak topology on ${\cal M}_1$.
%{\color{magenta}In folg.\ Def.\ in Anlehnung an \cite{Hampel1971,Cuevas1988} auch ``robust at $\mu$'' eingeführt.}

\begin{definition}\label{definition of qualitative robustness}
For a given set $M\subseteq{\cal M}$ and $\mu\in M$, the sequence of estimators $(\widehat T_n)$ is said to be {\em $M$-robust at $\mu$} if for every $\varepsilon>0$ there exists an open neighborhood $U=U(\mu,\varepsilon;M)$ of $\mu$ for the relative weak topology ${\cal O}_{\rm w}\cap M$ such that
$$%\begin{equation}\label{Robustness - characteristic property - Intro}
    \nu\in U\quad\Longrightarrow\quad \pi(\pr^{\mu}\circ\widehat T_n^{\,-1},\pr^{\nu}\circ\widehat T_n^{\,-1})\le\varepsilon \quad\mbox{for all }n\in\N.
$$%\end{equation}
The sequence $(\widehat T_n)$ is said to be {\em robust on $M$} if it is $M$-robust at every $\mu\in M$.
\end{definition}

In their pioneer work, Hampel \cite{Hampel1971} and Cuevas \cite{Cuevas1988} used (mainly the first part of) Definition \ref{definition of qualitative robustness} with specifically $M={\cal M}={\cal M}_1$ and established several criteria for robustness; cf.\ Theorems 1--2 in \cite{Hampel1971} and Theorems 1--2 in \cite{Cuevas1988}. In the present paper, our focus will be on the second part of Definition \ref{definition of qualitative robustness}, i.e.\ on robustness of $(\widehat T_n)$ on subsets $M$ of ${\cal M}$. In this context the following two criteria are already known for $M={\cal M}$.
\begin{itemize}
    \item[(I)] If $T:{\cal M}\rightarrow\Sigma$ is continuous for the relative weak topology ${\cal O}_{\rm w}\cap{\cal M}$, then $(\widehat T_n)$ is robust on ${\cal M}$.
    \item[(II)] If $(\widehat T_n)$ is weakly consistent and robust on ${\cal M}$, then $T:{\cal M}\rightarrow\Sigma$ is continuous for the relative weak topology ${\cal O}_{\rm w}\cap{\cal M}$.
\end{itemize}
Assertion (I) is a straightforward generalization of Theorem 2 in \cite{Cuevas1988} (where the author assumed ${\cal M}={\cal M}_1$) and assertion (II) is a special case of Theorem 1 in \cite{Cuevas1988}.

Recall that we assumed the set $\mathfrak{E}$ of all empirical probability measures to be contained in ${\cal M}$. As $\mathfrak{E}$ is dense in ${\cal M}_1$ w.r.t.\ the weak topology ${\cal O}_{\rm w}$ (cf.\ Theorem A.38 in \cite{FoellmerSchied2011} reformulated for probability measures), this implies that weak continuity of the map $T:{\cal M}\rightarrow\Sigma$ is a relatively strict requirement. For instance, in the case $E=\R$ the mean functional $T(\mu):=\int x\,\mu(dx)$ is not weakly continuous on $\mathfrak{E}$ (indeed, letting $x_{n,1}:=n$ and $x_{n,i} = 0$ for $i=2,\ldots,n$ and $n\in\N$, the sequence $(\widehat{m}_n(x_{n1},\dots,x_{nn}))_{n\in\N}$ converges to $\delta_{0}$ w.r.t. ${\cal O}_{\rm w}$, but $\int x\,\widehat{m}(x_{n1},\dots,x_{nn})(dx) = 1 \not\to 0 \int x\,\delta_{0}(dx)$). In view of (I)--(II), this simple example indicates that there are only a few relevant statistical functionals $T:{\cal M}\rightarrow\Sigma$ for which the corresponding sequence of estimators $(\widehat T_n)$ is robust on the whole domain ${\cal M}$. Nevertheless, for general statistical functionals one might ask for those subsets $M$ of ${\cal M}$ on which robustness of $(\widehat T_n)$ holds. The following simple example shows that this question can be reasonable.

\begin{examplenorm}\label{example intro exponential distriution}
Let $E=(0,\infty)$ and ${\cal E}$ be the class of all exponential distributions with mean $\theta$ (cf.\ Example \ref{parametric examples for w-sets - gamma}), $\theta\in(0,\infty)$. The unique maximum likelihood estimator for the parameter $\theta$ is known to be $\widehat T_n(\boldsymbol{x})=\overline{\boldsymbol{x}}_n$, where $\overline{\boldsymbol{x}}_n=\frac{1}{n}\sum_{i=1}^nx_i$ for $\boldsymbol{x}=(x_1,x_2,\ldots)$. It can be represented by $\widehat T_n(\boldsymbol{x})=T(\widehat m_n(\boldsymbol{x}))$ for the functional $T(\mu)=\int x\,\mu(dx)$ with domain ${\cal M}=\{\mu\in{\cal M}_1:\int |x|\,\mu(dx)<\infty\}$ and state space $\Sigma=(0,\infty)$. Since $T$ is weakly consistent by the law of large numbers but {\em not} weakly continuous on ${\cal M}$, assertions (I)--(II) imply that the sequence $(\widehat T_n)$ is not robust on ${\cal M}$. However, in Subsection \ref{Exponentialverteilung} we will see that our results yield robustness of $(\widehat T_n)$ on relatively large subsets of ${\cal M}$, in particular on ${\cal E}$. That means that the maximum likelihood estimator is robust at least against small deviations within the underlying parametric set of distributions ${\cal E}$. This statement could not be derived in the conventional theory of robustness. Note that robustness on ${\cal E}$ is of interest if one starts from the premise that both the target distribution $\mu$ and the distribution $\nu$ underlying the observations lie in ${\cal E}$.

In fact the maximum likelihood estimator is even robust against certain deviations out of ${\cal E}$, even though not against arbitrary deviations within the whole domain ${\cal M}$. For instance, at the end of Subsection \ref{Exponentialverteilung} we will see that the maximum likelihood estimator is also robust on the broader class $\Gamma$ of all Gamma distributions (with location parameter $0$). Robustness on $\Gamma\supsetneq{\cal E}$ is of interest if one assumes that the target distribution lies in ${\cal E}$ (so that the maximum likelihood principle is reasonable) but the distribution $\nu$ underlying the observations may lie in the broader class $\Gamma$.
{\hspace*{\fill}$\Diamond$\par\bigskip}
\end{examplenorm}

The issue of robustness on subsets was approached in Section 3.1 of \cite{Zaehle2016}. The latter paper develops further the theory of \cite{Kraetschmeretal2012,Kraetschmeretal2014} and provides the following criteria (i)--(ii), where ${\cal M}$ is assumed to be contained in the set ${\cal M}_1^{\psi}$ of all $\mu\in{\cal M}_1$ with $\int\psi\,d\mu<\infty$ for some given continuous function $\psi:E\to[0,\infty)$. By {\em $\psi$-weak topology} ${\cal O}_{\psi}$ on ${\cal M}_1^{\psi}$ we mean the coarsest topology for which all mappings $\mu\mapsto\int f\,d\mu$, $f\in {\cal C}_{\psi}$ are continuous, where ${\cal C}_{\psi}$ refers to the space of all continuous functions $f:E\rightarrow\R$ for which $\sup_{x\in E}|f(x)/(1+\psi(x))|<\infty$. For $E=\R^d$ and $\psi(x):=\|x\|^p$ with $  p\geq 1$, the set ${\cal M}_1^{\psi}$ is just the set of all Borel probability measures on $E=\R^d$ with finite $p$-th absolute moment and the $\psi$-weak topology is metrizable by the $L^p$-Wasserstein metric; cf.\ Lemma 8.3 in \cite{BickelFreedman1981}.

\begin{itemize}
    \item[(i)] If $T:{\cal M}\rightarrow\Sigma$ is continuous for the relative $\psi$-weak topology ${\cal O}_\psi\cap\cal{\cal M}$, then $(\widehat T_n)$ is robust on every subset $M\subseteq{\cal M}$ with ${\cal O}_\psi\cap M={\cal O}_{\rm w}\cap M$.
    \item[(ii)] If $(\widehat T_n)$ is weakly consistent on ${\cal M}$ and robust on every subset $M\subseteq{\cal M}$ with ${\cal O}_\psi\cap M={\cal O}_{\rm w}\cap M$, then $T:{\cal M}\rightarrow\Sigma$ is continuous for the $\psi$-weak topology ${\cal O}_\psi$.
\end{itemize}

In general the $\psi$-weak topology ${\cal O}_{\psi}$ is finer than the relative weak topology ${\cal O}_{\rm w}\cap{\cal M}_1^{\psi}$, and the two topologies coincide for $\psi\equiv 1$. Thus the criteria (i)--(ii) generalize the criteria (I)--(II). Assertion (i) says that for $\psi$-weakly continuous functionals $T:{\cal M}\rightarrow\Sigma$ the sequence $(\widehat T_n)$ is robust on every subset $M$ of ${\cal M}$ for which the relative $\psi$-weak topology ${\cal O}_\psi\cap M$ and the relative weak topology ${\cal O}_{\rm w}\cap M$ coincide. Lemma 3.4 of \cite{Zaehle2016} provides the following characterization of those subsets $M$ of ${\cal M}_1^\psi$ for which ${\cal O}_\psi\cap M={\cal O}_{\rm w}\cap M$ holds: the latter identity holds if and only if $M$ is locally uniformly $\psi$-integrating in the sense of Definition \ref{Def Uniformly Psi Integrating} below. On the one hand, this characterization is the basis for the proof of the criteria (i)--(ii) and is also relevant for robustness of more general estimators than plug-in estimators as defined in (\ref{Def Plug-in Estimator}); see \cite{Schuhmacheretal2015} for an example. On the other hand, the condition ``locally uniformly $\psi$-integrating'' is rather technical and not at all useful for checking the identity ${\cal O}_\psi\cap M={\cal O}_{\rm w}\cap M$ for any given set $M$. The aim of the present paper is to provide more useful characterizations of those subsets $M$ of ${\cal M}_1^\psi$ for which the identity ${\cal O}_\psi\cap M={\cal O}_{\rm w}\cap M$ holds, and to illustrate their use. For the sake of brevity we will refer to any $M\subseteq{\cal M}_1^\psi$ satisfying the condition ${\cal O}_\psi\cap M={\cal O}_{\rm w}\cap M$ as {\em w-set in ${\cal M}_1^\psi$}.

Theorem \ref{Key Prop} below gives three further equivalent conditions for a set to be a w-set in ${\cal M}_1^\psi$. Based on this theorem, we will obtain in Section~\ref{Sec Examples for locally uniformly intetgarting sets} several specific examples for w-sets in ${\cal M}_1^\psi$ for various choices of $\psi$. Among others, we will investigate popular parametric families of distribution such as normal, Pareto, Gumbel, or Gamma distributions, and also consider sets of distributions derived from Fr\'{e}chet classes of univariate marginal distributions via aggregation operators like the sum. The latter sets of distributions are of particular interest in the context of risk assessment. The results of Section~\ref{Sec Examples for locally uniformly intetgarting sets} together with assertion (i) above (and the results of Section~\ref{Sec Examples for continuous functionals}) in particular justify Example \ref{example intro exponential distriution}. In Section~\ref{Sec Examples for continuous functionals} we will provide examples for $\psi$-weakly continuous functionals $T$; we will study statistical functionals underlying the maximum likelihood method and law-invariant convex risk measures. In Section~\ref{Sec Examples for continuous functionals} we will also discuss the property of robustness of the corresponding plug-in estimators $(\widehat T_n)$ on subsets of $T$'s domain. Sections~\ref{Sec Proofs of main results} and~\ref{Sec Remaining proofs} contain longer proofs of our results.

Finally, note that we will in fact work with a slightly more general topology than ${\cal O}_\psi$, namely with the so-called $(\psi_k)$-weak topology ${\cal O}_{(\psi_k)}$ to be introduced at the beginning of Section \ref{Sec comparison of weak topologies}. This generalization does not have priority, but the respective theory covers some more examples than the theory for the $\psi$-weak topology ${\cal O}_\psi$. In particular, we need to establish a corresponding extension of the criteria (i)--(ii), which can be found in Theorem \ref{Key Prop - Corollary}.

%%%%%%%%%%%%%%%%%%%%%%%%%%%%%%%%%%%%%%%%%%%%%%%%%%%%%%%%%%%%%%%%
%%%%%%%%%%%%%%%%%%%%%%%%%%%%%%%%%%%%%%%%%%%%%%%%%%%%%%%%%%%%%%%%
%%%%%%%%%%%%%%%%%%%%%%%%%%%%%%%%%%%%%%%%%%%%%%%%%%%%%%%%%%%%%%%%
%%%%%%%%%%%%%%%%%%%%%%%%%%%%%%%%%%%%%%%%%%%%%%%%%%%%%%%%%%%%%%%%
%%%%%%%%%%%%%%%%%%%%%%%%%%%%%%%%%%%%%%%%%%%%%%%%%%%%%%%%%%%%%%%%
%%%%%%%%%%%%%%%%%%%%%%%%%%%%%%%%%%%%%%%%%%%%%%%%%%%%%%%%%%%%%%%%

\section{Concurrence of weak and $(\psi_k)$-weak topologies  and applications in robust statistics}\label{Sec comparison of weak topologies}

As before let $E$ be a Polish space and use the notation introduced in Section \ref{Introduction}. Let $(\psi_k)$ be a sequence of {\em gauge functions}, i.e., a sequence of continuous functions $\psi_k:E\rightarrow[0,\infty)$. Let ${\cal C}_{\psi_k}$ be the space of all continuous functions $f:E\rightarrow\R$ for which $\sup_{x\in E}|f(x)/(1+\psi_k(x))|<\infty$. Let ${\cal M}_1^{(\psi_k)}$ be the set of all Borel probability measures $\mu$ on $E$ for which $\int\psi_k\,d\mu<\infty$ for every $k\in\N$.  The {\em $(\psi_k)$-weak topology} ${\cal O}_{(\psi_k)}$ on ${\cal M}_1^{(\psi_k)}$ is defined to be the coarsest topology for which all mappings $\mu\mapsto\int f\,d\mu$, $f\in {\cal C}_{\psi_k}$, $k\in\N$, are continuous. When $\psi_k=\psi$ for all $k\in\N$, we have ${\cal M}_1^{(\psi_k)}={\cal M}_1^\psi$ and ${\cal O}_{(\psi_k)}={\cal O}_\psi$.

\begin{lemma}\label{Charakterisierung psi-schwache Topologie}
The set ${\cal M}_1^{(\psi_k)}$ equipped with the $(\psi_k)$-weak topology is a Polish space. In addition the $(\psi_{k})$-weak topology is metrizable by the metric
\begin{equation}\label{Det psi k metric}
    d_{(\psi_k)}(\mu,\nu) := d_{\rm w}(\mu,\nu) + \sum_{k=1}^\infty 2^{-k}\Big(\Big|\int \psi_k\,d\mu - \int \psi_k\,d\nu\Big|\wedge 1\Big),
\end{equation}
where $d_{\rm w}$  is any metric for the weak topology. Moreover, for every $(\mu_{n})_{n\in\N_0}\subset{\cal M}_1^{(\psi_k)}$ the following statements are equivalent.
\begin{enumerate}
    \item $\mu_{n}\to\mu_{0}$ $(\psi_k)$-weakly.
    \item $\mu_{n}\to\mu_{0}$ weakly and $\int\psi_k\,d\mu_{n}\to\int\psi_k\,d\mu_{0}$ for every $k\in\N$.
\end{enumerate}
\end{lemma}

In Theorem~\ref{Key Prop} below we will specify those subsets of ${\cal M}_1^{(\psi_{k})}$ on which the relative $(\psi_{k})$-weak topology and the relative weak topology  coincide. We will use the following terminology, which extends Definition 2.12 in \cite{Kraetschmeretal2014} and Definition 3.1 in \cite{Zaehle2016}.

\begin{definition}\label{Def Uniformly Psi Integrating}
A set $M\subseteq{\cal M}_1$ is said to be locally uniformly $(\psi_{k})$-integrating if for every $\mu\in M$, $\varepsilon>0$, and $k\in\N$ there exist  $a>0$ and a weakly open neighborhood $U$ of $\mu$  such that
$$
    \nu\in M\cap U\quad\Longrightarrow\quad \int \psi_{k}\eins_{\{\psi_{k}\ge a\}}\,d\nu\,\le\,\varepsilon.
$$
The set $M$ is said to be uniformly $(\psi_{k})$-integrating if for every $\varepsilon>0$ and $k\in\N$ there exists some $a>0$ such that
$$
    \sup_{\mu\in M}\int \psi_{k}\eins_{\{\psi_{k}\ge a\}}\,d\mu\,\le\,\varepsilon.
$$
If $(\psi_{k})$ consists of a single gauge function, say $\psi$, we shall speak of (locally) uniformly $\psi$-integrating sets instead of (locally) uniformly
$(\psi_{k})$-integrating sets.
\end{definition}

Note that a set $M$ is (locally) uniformly $(\psi_k)$-integrating if and only if it is (locally) uniformly $\psi_k$-integrating for every $k\in\N$. Of course, any uniformly $(\psi_{k})$-integrating set $M$ is also locally uniformly $(\psi_{k})$-integrating, and any locally uniformly $(\psi_{k})$-integrating set $M$ is a subset of ${\cal M}_1^{(\psi_{k})}$. If all $\psi_{k}$ are bounded, then the set ${\cal M}_1$ coincides with ${\cal M}_1^{(\psi_{k})}$ and is uniformly $(\psi_{k})$-integrating.

Let us now turn to the characterization of those subsets of ${\cal M}_1^{(\psi_{k})}$ on which the relative $(\psi_{k})$-weak topology and the relative weak topology coincide. For $\psi$-weak topologies, the equivalence (a)$\Leftrightarrow$(b) in the following theorem is already known from Lemma 3.6 in \cite{Zaehle2016}.

\begin{theorem}\label{Key Prop}
Let $(\psi_{k})$ be any sequence of gauge functions and $M\subseteq{\cal M}_1^{(\psi_{k})}$ be given. Then the following conditions are equivalent:
\begin{enumerate}
    %\item The relative weak and the relative $(\psi_{k})$-weak topologies on $M$ coincide.
    \item ${\cal O}_{(\psi_k)}\cap M={\cal O}_{\rm w}\cap M$.
    \item $M$ is locally uniformly $(\psi_{k})$-integrating.
    \item Every weakly compact subset of $M$ is uniformly $(\psi_{k})$-integrating.
    \item Every sequence in $M$ which converges weakly in $M$ is uniformly $(\psi_{k})$-integrating.
    \item For every sequence $(\mu_n)\subseteq M$ for which $\mu_n$ converges weakly to $\mu_0$ the convergence $\int\psi_k\,d\mu_n\to\int\psi_k\,d\mu_0$ holds for all $k\in\N$.
   \end{enumerate}
\end{theorem}

\begin{definition}\label{Def w-set}
Let $(\psi_{k})$ be any sequence of gauge functions and $M\subseteq{\cal M}_1^{(\psi_{k})}$. Then $M$ is said to be a w-set in ${\cal M}_1^{(\psi_{k})}$ if condition (a) (and thus each of the equivalent conditions (a)--(e)) in Theorem~\ref{Key Prop} holds.
\end{definition}

\begin{remarknorm}\label{dominance}
%{\color{blue}(i) The union of finitely many w-sets in ${\cal M}_1^{(\psi_{k})}$ is again a w-set in ${\cal M}_1^{(\psi_{k})}$, which can easily be derived from part (d) of Theorem \ref{Key Prop}.} {\color{magenta}Oder übersehe ich etwas?} \green{Ich denke, das passt so.} {\color{red}Das sehe ich noch nicht: z.B. Folge $(\mu_{n})$ in einer w-Menge $M_{1}$ schwach konvergent gegen $\mu$ in einer w-Menge $M_{2}$, aber $\mu\not\in M_{1}$. Was nun? Variation des Problems: Nehmt zwei disjunkte w-Mengen $M_{1}$ sowie $M_{2}$, und nehmt an, es existiert ein $\mu\in M_{1}$, welches au{\ss}erdem im topologischen Abschlu{\ss} von $M_{2}$ bzgl. der schwachen Topologie liegt. Dann k\"onnen wir zwar aus der w-Mengen Eigenschaft von $M_{1}$ eine offene Umgebung $U$ von $\mu$ finden, so dass $M_{1}\cap U$ klein genug ist, aber \"uber $M_{2}\cap U$ wissen wir a priori nichts. Oder?
%Wirklich einfach erscheint mir Anwendung von (a): Seien $M_{1},\dots,M_{r}$ w-Mengen und $G\in {\cal O}_{(\psi_{k})}$. Nach (a) existieren $H_{1},\dots,H_{r}\in {\cal O}_{w}$ mit
%$$
%G\cap M_{i} = H_{i}\cap M_{i} = H_{i}\cap \bigcup_{j=1}^{r}
%$$
%$$
%G \cap \bigcup_{i=1}^{r} M_{i} =  \bigcup_{i=1}^{r} M_{i}\cap G\in {\cal O}
%$$
%}
%(ii)
Let $(\psi_{k})$ and $(\widetilde{\psi}_{k})$ be sequences of gauge functions satisfying $\widetilde{\psi}_{k}\leq\psi_{k}$ pointwise for every $k\in\N$. Then ${\cal M}_1^{(\psi_{k})}\subseteq {\cal M}_1^{(\widetilde{\psi}_{k})},$ and the $(\psi_{k})$-weak topology is finer than the $(\widetilde{\psi}_{k})$-weak topology on ${\cal M}_1^{(\psi_{k})}$. In particular, every w-set in ${\cal M}_1^{(\psi_{k})}$ is also a w-set in ${\cal M}_1^{(\widetilde{\psi}_{k})}$. Moreover, if $\psi_k\equiv 1$ for every $k\in\N$, then every subset of ${\cal M}_1={\cal M}_1^{(\psi_k)}$ is a w-set.
{\hspace*{\fill}$\Diamond$\par\bigskip}
\end{remarknorm}

%\begin{remarknorm}\label{union of w-sets}
%
%{\hspace*{\fill}$\Diamond$\par\bigskip}
%\end{remarknorm}

We obtain the following generalization of Hampel's theorem, where by weak consistency of $(\widehat T_n)$ on ${\cal M}$ we mean that $\lim_{n\to\infty}\pr^\mu[|\widehat T_n-T(\mu)|\ge\eta]=0$ for all $\eta>0$ and $\mu\in {\cal M}$.

\begin{theorem}\label{Key Prop - Corollary}
Let the statistical model $(\Omega,{\cal F},\{\pr^\mu:\mu\in{\cal M}\})$ (with ${\cal M}\subseteq{\cal M}_1^{(\psi_k)}$), the functional $T:{\cal M}\rightarrow\Sigma$, and the sequence of estimators $(\widehat T_n)$ be as introduced in Section~\ref{Introduction}. Then the following two assertions hold:
\begin{itemize}
    \item[{\rm (i)}] If for any w-set $M$ in ${\cal M}_1^{(\psi_k)}$ with $M\subseteq{\cal M}$ the functional $T$ is continuous at every $\mu\in M$ for the relative $(\psi_k)$-weak topology ${\cal O}_{(\psi_k)}\cap{\cal M}$, then $(\widehat T_{n})$ is robust on $M$. In particular, if $T$ is continuous for the relative $(\psi_k)$-weak topology ${\cal O}_{(\psi_k)}\cap{\cal M}$, then $(\widehat T_n)$ is robust on every w-set $M$ in ${\cal M}_1^{(\psi_k)}$ with $M\subseteq{\cal M}$.
    \item[{\rm(ii)}] If $(\widehat T_n)$ is weakly consistent on ${\cal M}$ and robust on every w-set $M$ in ${\cal M}_1^{(\psi_k)}$ with $M\subseteq{\cal M}$, then $T$ is continuous for the relative $(\psi_k)$-weak topology ${\cal O}_{(\psi_k)}\cap{\cal M}$.
\end{itemize}
\end{theorem}

In the case where $\psi_k=\psi$ for all $k\in\N$, Theorem~\ref{Key Prop - Corollary} is already known from Theorem 3.8 in \cite{Zaehle2016}; previous versions of this result were obtained in \cite{Kraetschmeretal2012,Kraetschmeretal2014}. %\sout{There, considerations are restricted to $(\psi_{k})$-weak topologies with identical gauge functions $\psi_{k}$.}
The extension in terms of the general notion of $(\psi_{k})$-weak topology is motivated by the example of maximum likelihood estimation which will be studied in Subsection \ref{Continuityminimumcontrastfunctionals}. In particular, in order to establish local robustness for the maximum likelihood estimator of the scale parameter of Gumbel distributions the possibility to use nonconstant sequences of gauge functions will prove to be convenient; cf.\ Subsection \ref{Gumbelverteilung} below. Likewise, it will turn out that the full generality of Theorem \ref{Key Prop - Corollary} is useful when investigating local robustness of certain law-invariant convex risk measures; cf.\ the discussion subsequent to Corollary \ref{StetigkeitRisikofunktional cor}.

In many situations the functional $T:{\cal M}\rightarrow\Sigma$ can be shown to be $(\psi_k)$-weakly continuous on the whole domain ${\cal M}$. In some cases, however, it is beneficial that in condition (i) of Theorem \ref{Key Prop - Corollary} we only require continuity of $T$ at every point of $M$. To give an example, let $E=\R$, $\psi_k=\psi\equiv1$ (hence ${\cal M}_1^{(\psi_k)}={\cal M}_1$), $\alpha\in(0,1)$, and $T:{\cal M}_1\rightarrow\R$ be the functional that assigns to a Borel probability measure its (lower) $\alpha$-quantile. This (quantile) functional $T$ is ($\psi$-) weakly continuous at every point of the set $M$ of all Borel probability measures with a unique $\alpha$-quantile, but not weakly continuous on ${\cal M}_1$.

%%%%%%%%%%%%%%%%%%%%%%%%%%%%%%%%%%%%%%%%%%%%%%%%%%%%%%%%%%%%%%%%
%%%%%%%%%%%%%%%%%%%%%%%%%%%%%%%%%%%%%%%%%%%%%%%%%%%%%%%%%%%%%%%%
%%%%%%%%%%%%%%%%%%%%%%%%%%%%%%%%%%%%%%%%%%%%%%%%%%%%%%%%%%%%%%%%
%%%%%%%%%%%%%%%%%%%%%%%%%%%%%%%%%%%%%%%%%%%%%%%%%%%%%%%%%%%%%%%%
%%%%%%%%%%%%%%%%%%%%%%%%%%%%%%%%%%%%%%%%%%%%%%%%%%%%%%%%%%%%%%%%
%%%%%%%%%%%%%%%%%%%%%%%%%%%%%%%%%%%%%%%%%%%%%%%%%%%%%%%%%%%%%%%%

\section{Examples of w-sets in ${\cal M}_1^{(\psi_k)}$}\label{Sec Examples for locally uniformly intetgarting sets} %locally uniformly $(\psi_k)$-integrating sets}

The following lemma provides a simple but general class of w-sets in ${\cal M}_1^{(\psi_k)}$.

\begin{lemma}\label{relatively compact w set lemma}
Every set $M\subseteq{\cal M}_1^{(\psi_k)}$ that is relatively compact for the $(\psi_k)$-weak topology is a w-set in ${\cal M}_1^{(\psi_k)}$.
\end{lemma}

\begin{proof}
It suffices to show that on $M$ weak convergence implies $(\psi_{k})$-weak convergence. So let us suppose that $(\mu_n)$ is a sequence in $M$ that converges weakly to some $\mu\in {\color{cyan}M}$. Then by $(\psi_k)$-weak compactness, every subsequence of $(\mu_n)$ has a subsequence that converges $(\psi_k)$-weakly toward some $\nu\in {\cal M}_1^{(\psi_k)}$. Since $(\psi_k)$-weak convergence implies weak convergence, we must have $\nu=\mu$. It hence follows that $\mu_n\to\mu$ also $(\psi_k)$-weakly which completes the proof.
\end{proof}

The preceding lemma has the following consequence.

\begin{proposition}\label{compactness prop}
Suppose that $(\psi_k)$ is a sequence of gauge functions such that all sets of the form $\{\psi_{k+1}\le n\psi_k\}$ are compact for $k,n\in\mathbb{N}$. Then ${\cal M}_1^{(\psi_k)}$ is $(\psi_k)$-weakly compact and thus itself a w-set.
\end{proposition}

\begin{proof}Clearly, ${\cal M}_1^{(\psi_k)}$ is $(\psi_k)$-weakly closed, while relative compactness follows from Lemma \ref{psiweakcompactness} (d) by taking $\phi_k:=\psi_{k+1}$. Lemma \ref{relatively compact w set lemma} finally gives that ${\cal M}_1^{(\psi_k)}$ is a w-set.
\end{proof}

The most interesting case is where $E$ is non-compact. In this case Proposition \ref{compactness prop} is not applicable to constant sequences of gauge functions (i.e.\ $\psi_k=\psi$ for all $k\in\N$), because then the sets $\{\psi\le n\psi\}=E$, $n\in\N$, are not compact. However, as an immediate consequence of condition (e) in Theorem \ref{Key Prop} we obtain the following alternative device.

\begin{proposition}\label{A word on our motivation - remark 20}
Let $\Theta$ be a topological space, and $\mu_\theta$ be an element of ${\cal M}_1^{(\psi_k)}$ for every $\theta\in\Theta$. Then the set ${\cal M}_{\Theta}:=\{\mu_\theta:\theta\in\Theta\}$ is a w-set in ${\cal M}_1^{(\psi_k)}$ if the following two conditions are satisfied:
\begin{enumerate}
    \item For every sequence $(\theta_n)_{n\in\N_0}$ in $\Theta$, weak convergence of $\mu_{\theta_n}$ to $\mu_{\theta_0}$ implies $\theta_n\to\theta_0$.
    \item For every sequence $(\theta_n)_{n\in\N_0}$ in $\Theta$, convergence of $\theta_n$ to $\theta_0$ implies $\int\psi_k\,d\mu_{\theta_n}\to\int\psi_k\,d\mu_{\theta_0}$ for all $k\in\N$.
\end{enumerate}
%{\hspace*{\fill}$\Diamond$\par\bigskip}
\end{proposition}

%%%%%%%%%%%%%%%%%%%%%%%%%%%%%%%%%%%%%%%%%%%%%%%%%%%%%%%%%%%%%%%%
%%%%%%%%%%%%%%%%%%%%%%%%%%%%%%%%%%%%%%%%%%%%%%%%%%%%%%%%%%%%%%%%
%%%%%%%%%%%%%%%%%%%%%%%%%%%%%%%%%%%%%%%%%%%%%%%%%%%%%%%%%%%%%%%%

\subsection{Parametric classes of distributions}\label{Sec Examples for locally uniformly intetgarting sets - param} %locally uniformly $(\psi_k)$-integrating sets}

In this section, we consider a few examples in which parametric classes of probability distributions belong to ${\cal M}_1^{(\psi_k)}$ for suitably chosen sequences $(\psi_k)$ satisfying the hypotheses of Proposition \ref{compactness prop} or Proposition \ref{A word on our motivation - remark 20}. Note that in view of Remark \ref{dominance}, the assertions in the following examples can be reformulated for many other (coarser) topologies; see the second part of Example \ref{parametric examples for w-sets - gamma} for an illustration.

\begin{examplenorm}\label{parametric examples for w-sets - normal}
%{\bf (Multivariate normal distributions).}
Let $E=\R^d$ equipped with the euclidean norm $\|\cdot\|$ and $\mathcal{N}$ be the class of all $d$-dimensional normal distributions $N(m,\Sigma)$, where $m\in\R^d$ and $\Sigma$ is a semidefinite $d\times d$ covariance matrix. If we let $\psi_k(x):=\exp(\lambda_k{ \|x\|}^{\alpha_k})$, where $\lambda_k\uparrow\infty$ and $\alpha_k\uparrow 2$, we have $\mathcal{N}\subseteq{\cal M}_1^{(\psi_k)}$, and Proposition \ref{compactness prop} yields that $\mathcal{N}$ is a w-set in ${\cal M}_1^{(\psi_k)}$.
{\hspace*{\fill}$\Diamond$\par\bigskip}
\end{examplenorm}

\begin{examplenorm}\label{parametric examples for w-sets - pareto}
For fixed parameters $\alpha>0$ and $x_{\min}>0$, let $\mathcal{P}_{\alpha,x_{\min}}$ be the class of type-1 Pareto distributions with shape parameter $a\ge\alpha$. That is, $\mathcal{P}_{\alpha,x_{\min}}$ consists of all Borel probability measures on $\mathbb{R}$ with Lebesgue density
$$
    f_a(x)=\frac{a}{x_{\min}}\Big(\frac{x_{\min}}{x}\Big)^{a+1}\eins_{[x_{\min},\infty)}(x)
$$
for some $a\ge\alpha$. If we let $\psi_k(x):=|x|^{p_k}$, where $p_k>0$ and $p_k\uparrow\alpha$, we have $\mathcal{P}_{\alpha,x_m}\subseteq{\cal M}_1^{(\psi_k)}$, and Proposition \ref{compactness prop} yields that $\mathcal{P}_{\alpha,x_{\min}}$ is a w-set in ${\cal M}_1^{(\psi_k)}$.
{\hspace*{\fill}$\Diamond$\par\bigskip}
\end{examplenorm}

\begin{examplenorm}\label{parametric examples for w-sets - gamma}
Let $\Gamma$ denote the class of all Gamma distributions with location parameter $0$. That is, $\Gamma$ is the class of all Borel probability measures on $(0,\infty)$
with Lebesgue density
$$
    f_{\kappa,\theta}(x)=\frac{x^{\kappa-1}e^{-x/\theta}}{\theta^\kappa\,\Gamma(k)} %\,\eins_{(0,\infty)}(x)
$$
for some $\theta,\kappa>0$. When taking $\psi_k(x):=x^k$ or $\psi_k(x):=e^{\lambda_k x^{\beta_k}}$ where $\lambda_k\uparrow\infty$ and $\beta_k\uparrow 1$, we have $\Gamma\subseteq{\cal M}_1^{(\psi_k)}$, and Proposition \ref{compactness prop} yields that $\Gamma$ is a w-set in ${\cal M}_1^{(\psi_k)}$.

When the parameter $\kappa$ is fixed and set to $1$, the Lebesgue density $f_{\kappa,\theta}$ simplifies to the Lebesgue density
$$
    f_{\theta}(x):=f_{1,\theta}(x)=\frac{e^{-x/\theta}}{\theta}  %\,\eins_{(0,\infty)}(x)
$$
of the exponential distribution to the parameter $\theta>0$, and the corresponding class of all exponential distributions will be denoted by ${\cal E}$. Again by Proposition \ref{compactness prop} we obtain that ${\cal E}$ is a w-set in ${\cal M}_1^{(\psi_k)}$ for the sequences $(\psi_k)$ of gauge functions mentioned above. In Subsection \ref{Exponentialverteilung} we will consider the single gauge function $\psi(x):=x$ which is dominated by $\psi_k(x):=e^{\lambda_k x^{\beta_k}}$ (with $\lambda_k,\beta_k$ as above) for every $k$. Thus ${\cal E}\subseteq{\cal M}_1^{(\psi_k)}\subseteq{\cal M}_1^{\psi}$, and by Remark \ref{dominance}
%{\color{red}mit Remark \ref{dominance} k\"onnen wir nur folgern, dass ${\cal E}$ eine w-Menge in ${\cal M}_1^{(\psi_{k})}$ ist} {\color{magenta}Wieso? Das verstehe ich nicht ganz.} {\color{cyan}by a routine application of Proposition \ref{A word on our motivation - remark 20}}
we obtain that ${\cal E}$ is also a w-set in ${\cal M}_1^{\psi}$.
%means of Proposition \ref{A word on our motivation - remark 20} we obtain that ${\cal E}$ is a w-set in ${\cal M}_1^{\psi}$. Indeed, if for some sequence $(\theta_{n})$ in $(0,\infty)$ the corresponding sequence $(\mu_{\theta_{n}})$ in ${\cal E}$ converges weakly to some $\mu_{\theta_0}\in {\cal E}$, then obviously $\theta_{n}\to\theta_{0}$ and $|x|f_{\theta_{n}}(x)\to |x|f_{\theta_{0}}(x)$ for every $x\in\R$. In addition $\theta_{n}\leq 2\theta_{0}$ for sufficient large $n$ so that for sufficient large $n$ the inequality $f_{\theta_{n}}(x)\leq f(x)$ holds for all $x\in\R$, where $f(x)$ is defined to take the values $0$, $2\theta_0$, $f_{2\theta_0}(x)$ according to whether $x<0$, $0\le x\le 2\theta_0$, or $x>2\theta_0$. Since ${\cal E}\subseteq\cM_1^{\psi}$, we have $\int |x| f(x)\,dx<\infty$, and then an application of the dominance convergence theorem yields $\int\psi\,d\mu_{\theta_{n}} = \int |x| f_{\theta_{n}}(x)\,dx\rightarrow\int |x| f_{\theta_{0}}(x)\,dx = \int\psi\,d\mu$. Hence conditions (a) and (b) of Proposition \ref{A word on our motivation - remark 20} are shown.
{\hspace*{\fill}$\Diamond$\par\bigskip}
\end{examplenorm}

\begin{examplenorm}\label{parametric examples for w-sets - gumbel}
Let $\mathcal{G}$ denote the class of all Gumbel distributions, i.e., the class of all Borel probability measures ${\rm G}_{a}$ on $\mathbb{R}$ with Lebesgue density
$$
    f_a(x)=ae^{-ax-e^{-ax}}
$$
for some $a>0$. By letting $\psi_k(x):=|x|^k$ or $\psi_k(x):=e^{\lambda_k|x|^{\beta_k}}$ where $\lambda_k\uparrow\infty$ and $\beta_k\uparrow 1$, we obtain that $\mathcal{G}\subseteq{\cal M}_1^{(\psi_k)}$, and Proposition \ref{compactness prop} yields that $\mathcal{G}$ is a w-set in ${\cal M}_1^{(\psi_k)}$.

In Subsection \ref{Gumbelverteilung} we will consider the gauge functions $\psi_k(x):=|\log a_k - a_kx - e^{-a_kx}|$ for some sequence $(a_{k})$ representing $(0,\infty)\cap\Q$. Since for $a > 0$ the moment generating function of ${\rm G}_{a}$ is well defined on $(-\infty,1/a)$ enclosing $0$, the integrals $\int\widetilde a|x|e^{-ax -e^{-ax}}dx$ and $\int e^{-\widetilde a x}  e^{-ax -e^{-ax}}dx$ are finite for any $a,\widetilde a> 0$, and thus ${\rm G}_{a}\in\cM_1^{(\psi_{k})}$ for all $a>0$. We now verify conditions (a)--(b) of Proposition~\ref{A word on our motivation - remark 20} to show that $\mathcal{G}$ is also a w-set in $\cM_1^{(\psi_{k})}$ for this choice of gauge functions.

(a): Let $({\rm G}_{a_n})_{n\in\N}$ be any sequence in ${\cal G}$ which weakly converges to some distribution ${\rm G}_{a_0}\in {\cal G}$. Then corresponding sequence $(F_{a_{n}})_{n\in\N_{0}}$ of distribution functions satisfies
$$
    e^{-e^{-a_{n} x}} = F_{a_{n}}(x) \longrightarrow  F_{a}(x) = e^{-e^{-ax}}\quad\mbox{for all }x\in\R.
$$
Thus necessarily $a_n\to a_0$.

(b): Let $(a_n)_{n\in\N}$ be a sequence in $(0,\infty)$ which converges to some $a_0\in(0,\infty)$. Set $\underline{a}:=\inf_{n\in\N}a_n$ and $\overline{a}:=\sup_{n\in\N}a_n$, and note that $\underline{a}>0$ and $\overline{a}<\infty$. For any $x\in\R$, the mapping $a\mapsto - ax - e^{-ax}$ is nonincreasing on $(0,\infty)$. Thus
$$
    \sup_{n\in\N}\,\psi_k(x)\,a_ne^{-a_{n}x -e^{-a_{n}x}}\le\,\psi_k(x)\,\overline{a}\,e^{-\underline{a}x -e^{-\underline{a}x}}\quad\mbox{ for every }x\in\R.
$$
Hence by ${\rm G}_{\overline{a}}\in\cM_1^{(\psi_{k})}$, we may apply the dominated convergence theorem to conclude that for every $k\in\N$,
\begin{eqnarray*}
    \lefteqn{\int\psi_k(x)\,{\rm G}_{a_n}(dx)\,=\,\int \psi_k(x)\,a_{n}e^{-a_{n}x -e^{-a_{n}x}}\, dx}\\
    & \longrightarrow & \int\psi_k(x)\,a_{0}e^{-a_{0}x -e^{-a_{0}x}}\, dx\,=\,\int\psi_k(x)\,{\rm G}_{a_0}(dx).
\end{eqnarray*}

As we have shown that conditions (a)--(b) of Proposition~\ref{A word on our motivation - remark 20} are satisfies, the proposition ensures that $\mathcal{G}$ is indeed a w-set in $\cM_1^{(\psi_{k})}$.
{\hspace*{\fill}$\Diamond$\par\bigskip}
\end{examplenorm}

%%%%%%%%%%%%%%%%%%%%%%%%%%%%%%%%%%%%%%%%%%%%%%%%%%%%%%%%%%%%%%%%
%%%%%%%%%%%%%%%%%%%%%%%%%%%%%%%%%%%%%%%%%%%%%%%%%%%%%%%%%%%%%%%%
%%%%%%%%%%%%%%%%%%%%%%%%%%%%%%%%%%%%%%%%%%%%%%%%%%%%%%%%%%%%%%%%

\subsection{Images of Fr\'{e}chet classes}\label{aggregationrobustness}

In this section we consider certain classes $M$ that consist of the image measures of the probability measures in a given Fr\'{e}chet class w.r.t.\ fixed one-dimensional marginal distributions. The motivation for studying this examples is the notion of aggregation robustness recently introduced by Embrechts et al.\ \cite{Embrechtsetal2014}; see also Subsection \ref{ContinuityRiskfunctionals} below.

Let  $E=\R^{r}$ be endowed with the Euclidean norm $\|\cdot\|$. For given Borel probability measures $\mu_{1},\dots,\mu_{d}$ on $\R$ denote by $\boldsymbol{{\cal M}}(d;\mu_{1},\dots,\mu_{d})$ the set of all Borel probability measures on $\R^d$ whose one-dimensional marginal distributions are  $\mu_{1},\dots,\mu_{d}$. The set $\boldsymbol{{\cal M}}(d;\mu_{1},\dots,\mu_{d})$ is sometimes called {\em Fr\'{e}chet class} associated with $\mu_1,\ldots,\mu_d$. For any Borel measurable map $A_{d}:\R^d\to\R^{r}$ let
\begin{equation}\label{Def M A d}
    M=M(\mu_{1},\dots,\mu_{d};A_{d}):=\big\{\boldsymbol{\mu}\circ A_{d}^{-1}:\,\boldsymbol{\mu}\in\boldsymbol{{\cal M}}(d;\mu_{1},\dots,\mu_{d})\big\}\subseteq{\cal M}_1
\end{equation}
be the class of the images under $A_{d}$ of all probability measures from $\boldsymbol{{\cal M}}(d;\mu_{1},\dots,\mu_{d})$. In the following Example \ref{Example aggregation maps} we give a few examples of maps $A_d$ we have in mind. All these maps are Lipschitz continuous and thus satisfy condition (a) of Proposition \ref{Aggregationsrobustheit} ahead. This proposition will provide a general criterion for determining whether $M(\mu_{1},\dots,\mu_{d};A_{d})$ is a w-set.

\begin{examplenorm}\label{Example aggregation maps}
A simple but relevant example for a map $A_d:\R^d\rightarrow\R^d$ is the identity map:
\begin{itemize}
    \item [(i)] $A_{d}(\boldsymbol{x}):=\boldsymbol{x}$.
\end{itemize}
In this case the set $M(\mu_{1},\dots,\mu_{d};A_{d})$ defined by (\ref{Def M A d}) is nothing but the Fr\'{e}chet class $\boldsymbol{{\cal M}}(d;\mu_{1},\dots,\mu_{d})$ itself. Examples for maps $A_{d}:\R^{d}\rightarrow\R$ that are relevant in risk management and insurance are given by
\begin{itemize}
    \item [(ii)] $A_{d}(x_{1},\dots,x_{d}) := \sum_{i=1}^{d}x_{i}$,
    \item [(iii)] $A_{d}(x_{1},\dots,x_{d}) := \max\{x_{1},\dots,x_{d}\}$,
    \item [(iv)] $A_{d}(x_{1},\dots,x_{d}) := \sum_{i=1}^{d}(x_{i} - t_{i})^{+}$ for thresholds $t_{1},\dots,t_{d} > 0$,
    \item [(v)] $A_{d}(x_{1},\dots,x_{d}) := (\sum_{i=1}^{d}x_{i} - t)^{+}$ for a threshold $t > 0$;
\end{itemize}
see, for instance, \cite[p.\,248]{McNeiletal2005}. For an application of (ii) see Subsection~\ref{ContinuityRiskfunctionals} below.
{\hspace*{\fill}$\Diamond$\par\bigskip}
\end{examplenorm}

Let us now turn over to our main criterion to check whether the set $M$ defined in \eqref{Def M A d} is a w-set in ${\cal M}_1^{(\widetilde\psi_k)}$. Here we use the notation $\widetilde\psi_k(\cdot):=\psi_k(\|\cdot\|)$. Note that $\widetilde\psi_k=\psi_k$, $k\in\N$, when $r=1$.

\begin{proposition}\label{Aggregationsrobustheit}
Let $A_d:\R^d\to\R^r$ be any Borel measurable map. Let $(\psi_k)$ be any sequence of gauge functions on $\R$ that are all convex and even, and consider the sequence of gauge functions $(\widetilde\psi_k)$ on $\R^r$ with $\widetilde\psi_k$ as above. Further assume that the following assertions hold:
\begin{enumerate}
    \item There exist constants $b,c>0$ such that $\|A_d(\boldsymbol{x})\|\le b+c\sum_{i=1}^d|x_i|$ for all $\boldsymbol{x}=(x_1,\ldots,x_d)\in\R^d$.
    \item For every $k\in\N$ there exist $\ell_k\in\N$ and $c_k>0$ such that $\psi_k((d+1)cx)\le c_k\psi_{\ell_k}(x)$ for all $x\in\R$, where the constant $c$ is given by (a).
    \item $\mu_{1},\dots,\mu_{d}\in\cM_{1}^{(\psi_k)}$.
%    \item $N$ is uniformly $(\psi_k)$-integrating.
\end{enumerate}
Then $M$ defined by (\ref{Def M A d}) is uniformly $(\widetilde\psi_k)$-integrating, and thus it is a w-set in ${\cal M}_1^{(\widetilde\psi_k)}$.
\end{proposition}

%%%%%%%%%%%%%%%%%%%%%%%%%%%%%%%%%%%%%%%%%%%%%%%%%%%%%%%%%%%%%%%%
%%%%%%%%%%%%%%%%%%%%%%%%%%%%%%%%%%%%%%%%%%%%%%%%%%%%%%%%%%%%%%%%
%%%%%%%%%%%%%%%%%%%%%%%%%%%%%%%%%%%%%%%%%%%%%%%%%%%%%%%%%%%%%%%%
%%%%%%%%%%%%%%%%%%%%%%%%%%%%%%%%%%%%%%%%%%%%%%%%%%%%%%%%%%%%%%%%
%%%%%%%%%%%%%%%%%%%%%%%%%%%%%%%%%%%%%%%%%%%%%%%%%%%%%%%%%%%%%%%%
%%%%%%%%%%%%%%%%%%%%%%%%%%%%%%%%%%%%%%%%%%%%%%%%%%%%%%%%%%%%%%%%

\section{Examples for $(\psi_k)$-weakly continuous functionals $T$}\label{Sec Examples for continuous functionals}

In this section, we will discuss continuity of some relevant statistical functionals w.r.t.\ the $(\psi_k)$-weak topology for suitable sequences of gauge functions $(\psi_k)$. In Section~\ref{Continuityminimumcontrastfunctionals} we will consider statistical functionals which underly the maximum likelihood estimation principle. In Section~\ref{ContinuityRiskfunctionals} we will consider statistical functionals associated with so-called risk measures. The latter play an important role in quantitative risk management; see, for instance, \cite{McNeiletal2005,Rueschendorf2013}.

%%%%%%%%%%%%%%%%%%%%%%%%%%%%%%%%%%%%%%%%%%%%%%%%%%%%%%%%%%%%%%%%
%%%%%%%%%%%%%%%%%%%%%%%%%%%%%%%%%%%%%%%%%%%%%%%%%%%%%%%%%%%%%%%%
%%%%%%%%%%%%%%%%%%%%%%%%%%%%%%%%%%%%%%%%%%%%%%%%%%%%%%%%%%%%%%%%

\subsection{Maximum likelihood functionals}\label{Continuityminimumcontrastfunctionals}

%{\color{magenta}Die ``Einleitung'' und die alten Abschnitte 4.1.1 und 4.1.2 habe ich wegen dem neunten Spiegelpunkt im Gutachten zusammengefasst. Ebenso habe ich das dominating measure von $\nu$ in $\lambda$ umgenannt, da $\nu$ unten auch als Beobactungsmodell auftritt; siehe Spiegelpunkte 12 und 13 im Gutachen. Ich hoffe, ich habe keine ``Ersetzung'' vergessen.\\}

Let $\Theta\subseteq\R^d$ and $\mu_\theta$ be a Borel probability measure on $E$ for every $\theta\in\Theta$. Let
$$
    (\Omega,{\cal F}):=(E^\N,{\cal B}(E)^{\otimes\N})\quad\mbox{ and }\quad \pr^\theta:=\mu_\theta^{\otimes\N}\mbox{ for every }\theta\in\Theta.
$$
Then $(\Omega,{\cal F},\{\pr^\theta:\theta\in\Theta\})$ is a parametric statistical product model. We will assume that our parametric statistical model is dominated. That is, we assume that there is some $\sigma$-finite measure $\lambda$ on $(E,{\cal B}(E))$ (called the dominating measure) such that for every $\theta\in\Theta$ the law $\mu_\theta$ is absolutely continuous w.r.t.\ $\lambda$ with Radon--Nikodym derivative
$$
    f_\theta\,:=\,\frac{d\mu_\theta}{d\lambda}\,.
$$
In particular, when $X_i$ denotes the $i$-th coordinate projection on $\Omega=E^\N$, the law $\pr^\theta\circ(X_1,\ldots,X_n)^{-1}=\mu_\theta^{\otimes n}$ of a sample of size $n$ is absolutely continuous w.r.t.\ $\lambda^{\otimes n}$ with Radon--Nikodym derivative $d\mu_\theta^{\otimes n}/d\lambda^{\otimes n}$ satisfying
$$
    \frac{d\mu_\theta^{\otimes n}}{d\lambda^{\otimes n}}(x_1,\ldots,x_n)\,=\,\prod_{i=1}^n f_\theta(x_i)=:\,L_n(x_1,\ldots,x_n;\theta)\quad\mbox{for all }x_1,\ldots,x_n\in E
$$
for every $\theta\in\Theta$ and $n\in\N$.

%%%%%%%%%%%%%%%%%%%%%%%%%%%%%%%%%%%%%%%%%%%%%%%%%%%%%%%%%%%%%%%%
%%%%%%%%%%%%%%%%%%%%%%%%%%%%%%%%%%%%%%%%%%%%%%%%%%%%%%%%%%%%%%%%

%\subsubsection{Maximum likelihood estimators and their functional representations}\label{Maximum likelihood estimators and functionals}

By definition, a maximum likelihood estimator $\widehat T_n$ for the parameter $\theta$ based on a sample of size $n$ satisfies
\begin{equation}\label{Def MLE}
    \widehat{T}_{n}(\boldsymbol{x})=\widehat{T}_{n}(x_1,\ldots,x_n)\in\argmax_{\theta\in\Theta} L_n(x_1,\ldots,x_n;\theta)\quad\mbox{ for all }\boldsymbol{x}=(x_1,x_2,\ldots)\in\Omega
\end{equation}
for every $\theta\in\Theta$. Let us assume that $f_\theta>0$ on $E$ for every $\theta\in\Theta$.
Then condition (\ref{Def MLE}) is equivalent to
$$%\begin{equation}\label{Def MLLE}
    \widehat{T}_{n}(\boldsymbol{x})=\widehat{T}_{n}(x_1,\ldots,x_n)\in\argmax_{\theta\in\Theta} \frac{1}{n}\sum_{i=1}^{n}\log f_{\theta}(x_{i})\quad\mbox{ for all }\boldsymbol{x}=(x_1,x_2,\ldots)\in\Omega.
$$%\end{equation}
In particular, if ${\cal M}\subseteq{\cal M}_1$ contains the set $\mathfrak{E}$ of all empirical probability measures and $T:{\cal M}\rightarrow\Theta$ is a functional satisfying
\begin{equation}\label{Def MLF}
    T(\mu)\in\argmax_{\theta\in\Theta} \int \log f_{\theta}\,d\mu\quad\mbox{ for all }\mu\in{\cal M},
\end{equation}
then we can define a maximum likelihood estimator by
\begin{equation}\label{Representation MLE}
    \widehat{T}_{n}(\boldsymbol{x})=\widehat{T}_{n}(x_1,\ldots,x_n)=T(\widehat m_n(x_1,\ldots,x_n))\quad\mbox{ for all }\boldsymbol{x}=(x_1,x_2,\ldots)\in\Omega.
\end{equation}

Inspired by the representation (\ref{Representation MLE}) of the maximum likelihood estimator we introduce the following terminology. For any subset $\cM\subseteq\cM_{1}$ containing $\mathfrak{E}$, a map $T:{\cal M}\rightarrow\Theta$ is called {\em maximum likelihood functional} associated with the statistical model $(\Omega,{\cal F},\{\pr^\theta:\theta\in\Theta\})$ if for every $\mu\in{\cal M}$ the integrals $\int |\log f_{\theta}|\,d\mu$, $\theta\in\Theta$, are finite and (\ref{Def MLF}) holds.

\begin{remarknorm}\label{ML Jensen Remark}
Although it is not crucial for our purposes, note that a maximum likelihood functional $T$ evaluated at $\mu_\theta$ takes the parameter $\theta$ as its value, i.e., $\theta$ provides a maximizer of the mapping $\vartheta\mapsto \int \log f_{\vartheta}\,d\mu_\theta$ on $\Theta$. Moreover, $\theta$ is the unique maximizer as soon as $\mu_\theta\neq\mu_\vartheta$ for all $\vartheta\in\Theta\setminus\{\theta\}$. Indeed, since we assumed $f_\theta>0$ for every $\theta\in\Theta$, the strict concavity of the logarithm and Jensen's inequality yield
\begin{eqnarray*}
    \int \log f_{\vartheta}\,d\mu_\theta - \int \log f_{\theta}\,d\mu_\theta=\int \log\Big(\frac{f_{\vartheta}}{f_{\theta}}\Big)d\mu_\theta <
    \log\Big(\int\frac{f_{\vartheta}}{f_{\theta}}\,d\mu_\theta\Big) = \log(1) = 0
\end{eqnarray*}
for every $\vartheta\in\Theta$ for which $f_\vartheta\neq f_\theta$; in the third step we have used the fact that $f_\vartheta/f_\theta=d\mu_\vartheta/d\mu_\theta$.
{\hspace*{\fill}$\Diamond$\par\bigskip}
\end{remarknorm}

%%%%%%%%%%%%%%%%%%%%%%%%%%%%%%%%%%%%%%%%%%%%%%%%%%%%%%%%%%%%%%%%
%%%%%%%%%%%%%%%%%%%%%%%%%%%%%%%%%%%%%%%%%%%%%%%%%%%%%%%%%%%%%%%%

%\subsubsection{A word on our motivation}\label{A word on our motivation}

In view of (\ref{Representation MLE}) the minimal domain of a maximum likelihood functional $T:{\cal M}\rightarrow\Theta$ is ${\cal M}=\mathfrak{E}$, where as before $\mathfrak{E}$ refers to the set of all empirical probability measures on $E$. In order to apply Theorem~\ref{Key Prop - Corollary} to maximum likelihood estimators, the domain ${\cal M}$ has to be chosen so large that it contains both $\mathfrak{E}$ and the set
$$
    {\cal M}_\Theta:=\{\mu_\theta:\,\theta\in\Theta\}.
$$
Indeed, if $\mathfrak{E}\cup{\cal M}_{\Theta}\subseteq{\cal M}\subseteq{\cal M}_1^{(\psi_k)}$ and a maximum likelihood functional $T:{\cal M}\rightarrow\Theta$ is $(\psi_k)$-weakly continuous for some sequence of gauge functions $(\psi_k)$, then Theorem~\ref{Key Prop - Corollary} shows that the corresponding sequence of maximum likelihood estimators $(\widehat T_n)$ is robust on every w-set $M$ in ${\cal M}_1^{(\psi_k)}$ with $M\subseteq{\cal M}$ (in particular with $M\subseteq {\cal M}_\Theta$). Recall that robustness of $(\widehat T_n)$ on $M$ means that for every $\mu\in M$ and $\varepsilon>0$ there exists an open neighborhood $U=U(\mu,\varepsilon;M)$ of $\mu$ for the relative weak topology ${\cal O}_{\rm w}\cap M$ such that $\pi(\pr^{\mu}\circ\widehat T_n^{\,-1},\pr^{\nu}\circ\widehat T_n^{\,-1})\le\varepsilon$ for all $\nu\in U$ and $n\in\N$.

\begin{remarknorm}\label{A word on our motivation - remark 10}
In many specific situations the set ${\cal M}_\Theta$ itself can be shown to be a w-set in ${\cal M}_1^{(\psi_k)}$; see Examples \ref{parametric examples for w-sets - normal} through \ref{parametric examples for w-sets - gumbel}. In this case the sequence $(\widehat T_n)$ of maxium likelihood estimators is robust on ${\cal M}_\Theta$.
{\hspace*{\fill}$\Diamond$\par\bigskip}
\end{remarknorm}

Note that robustness on ${\cal M}_\Theta$ is of interest if one starts from the premise that both the target distribution $\mu$ and the distribution $\nu$ underlying the observations lie in the parametric class of distributions ${\cal M}_\Theta$. On the other hand, robustness on $M\supsetneq{\cal M}_\Theta$ is of interest if one assumes that the target distribution lies in ${\cal M}_\Theta$ (so that the maximum likelihood principle is reasonable) but the distribution $\nu$ underlying the observations may lie in a broader class $M$ (nested between ${\cal M}_\Theta$ and ${\cal M}$).
\medskip

%%%%%%%%%%%%%%%%%%%%%%%%%%%%%%%%%%%%%%%%%%%%%%%%%%%%%%%%%%%%%%%%
%%%%%%%%%%%%%%%%%%%%%%%%%%%%%%%%%%%%%%%%%%%%%%%%%%%%%%%%%%%%%%%%

\subsubsection{$(\psi_k)$-weak continuity of maximum likelihood functionals}\label{Continuity of maximum likelihood functionals}

Here we are going to investigate a fixed maximum likelihood functional $T:\cM\rightarrow\Theta$ for continuity w.r.t.\ a suitable $(\psi_k)$-weak topology. To this end, let $\Theta$ be an open convex subset of $\R^{d}$ equipped with the Euclidean norm $\|\cdot\|$. We will assume throughout that the following three conditions hold:
\begin{eqnarray}
    & & f_{\theta}(x) > 0\quad\mbox{ for every }\theta\in\Theta\mbox{ and }x\in E. \label{interior1} \\
    & & x\longmapsto \log f_{\theta}(x)\quad \mbox{is continuous for every }\theta\in\Theta. \label{continuity1} \\
    & & \theta\longmapsto \log f_{\theta}(x)\quad \mbox{is concave for every }x\in E. \label{convexity1}
\end{eqnarray}
In general we cannot expect that the maximum likelihood functional $T$ is weakly continuous. On the other hand, it should often be possible to find a sequence of gauge functions $(\psi_k)$ for which $T$ is $(\psi_k)$-weakly continuous. At this abstract level the most obvious candidates for the sequence $(\psi_{k})$ is built upon a sequence $(\log f_{\theta_{k}}(\cdot))$ with $(\theta_{k})$ a sequence in $\Theta$ representing $\Theta\cap\Q^{d}$. For instance, we may and do choose
\begin{equation}\label{Def psi k}
    \psi_k(x):= |\log f_{\theta_k}(x)|,\quad x\in E.
\end{equation}
For every $\mu\in\cM_{1}$ with $\int|\log f_{\theta}|\,d\mu<\infty$ for all $\theta\in\Theta$ we can now define a map ${\cal L}_{\mu}:\Theta\rightarrow\R$ by
$$
    {\cal L}_{\mu}(\theta) := \int \log f_{\theta}\,d\mu.
$$
Note that ${\cal L}_\mu$ is concave by (\ref{convexity1}) and thus continuous. %finite on $\Theta,$ in particular $L_{\mu}|_{\Theta}$ is continuous.
Moreover, $T(\mu)$ is a maximum point of ${\cal L}_{\mu}$ by the characteristic property (\ref{Def MLF}) of the functional $T$.
\medskip

\begin{theorem}\label{contrastfunction}
Let ${\cal M}\subseteq{\cal M}_1$ and $T:{\cal M}\rightarrow\Theta$ be a maximum likelihood functional. Assume that (\ref{interior1})--(\ref{convexity1}) hold, and let the function $\psi_k$ be defined by (\ref{Def psi k}) for every $k\in\N$. Let $\mu\in\cM$ be such that the map ${\cal L}_{\mu}$ has a unique maximum point. Then $T(\mu_{n})\to T(\mu)$ for any sequence $(\mu_{n})\subset\cM$ which satisfies ${\cal L}_{\mu_{n}}(\theta_{k})\to{\cal L}_{\mu}(\theta_{k})$ for every $k\in\N$. In particular, $T$ is continuous at $\mu$ w.r.t.\ the $(\psi_{k})$-weak topology.
\end{theorem}

\begin{remarknorm}\label{reduction to single}
Let the gauge functions $\psi_{k}$ be defined as in (\ref{Def psi k}). Furthermore, let $k_0\in\N$ such that for every $k\in\N$ there exist constants $B_k\ge 0$ and $C_k>0$ with $\psi_k\le B_k+C_k\psi_{k_0}$. So $\overline{\psi}_{k}:=B_{k}+C_{k}\psi_{k_{0}}$ defines a new sequence $(\overline{\psi}_{k})$ of gauge functions such that
$\cM_1^{\psi_{k_{0}}} = \cM_1^{(\overline{\psi}_{k})} = \cM_1^{(\psi_{k})}$. Then on the one hand by the genuine definition of $(\psi_{k})$-weak topology, the $(\overline{\psi}_{k})$-weak topology is finer than the $(\psi_{k})$-weak topology. On the other hand by Lemma \ref{Charakterisierung psi-schwache Topologie} the $(\overline{\psi}_{k})$-weak topology is coarser than the $\psi_{k_{0}}$-weak topology. In particular the $(\psi_{k})$-topology coincides with the $\psi_{k_{0}}$-weak topology so that we may replace in Theorem \ref{contrastfunction} the $(\psi_{k})$-weak topology with the $\psi_{k_{0}}$-weak topology.
%This way of reduction of a sequence of gauge functions to a single one of them is not always possible; cf.\ the example in Subsection~\ref{Gumbelverteilung} below.
{\hspace*{\fill}$\Diamond$\par\bigskip}
\end{remarknorm}

We round out Section~\ref{Continuityminimumcontrastfunctionals} with two specific examples illustrating Theorem~\ref{contrastfunction}.

%%%%%%%%%%%%%%%%%%%%%%%%%%%%%%%%%%%%%%%%%%%%%%%%%%%%%%%%%%%%%%%%
%%%%%%%%%%%%%%%%%%%%%%%%%%%%%%%%%%%%%%%%%%%%%%%%%%%%%%%%%%%%%%%%

\subsubsection{Example exponential distribution}\label{Exponentialverteilung}

Let specifically $E=(0,\infty)$, $\Theta=(0,\infty)$ and ${\cal M}_\Theta:=\{{\rm Exp}_{\theta}:\,\theta\in(0,\infty)\}$ be the class of all exponential distributions ${\rm Exp}_{\theta}$  with parameter $\theta$. In this case we have
\begin{equation}\label{Exponentialverteilung - 10}
    \log f_\theta(x)=-\log\theta - x/\theta \quad\mbox{ for all }\theta\in(0,\infty)\mbox{ and }x\in (0,\infty).
\end{equation}
If we choose the sequence of gauge functions $(\psi_k)$ by (\ref{Def psi k}), then we obtain
\begin{equation}\label{Exponentialverteilung - 20}
    \psi_k(x)= |\log f_{\theta_k}(x)|= |-\log \theta_k - x/\theta_k|\quad \mbox{ for all }x\in (0,\infty)\mbox{ and }k\in\N
\end{equation}
for some sequence $(\theta_{k})$ in $\Theta=(0,\infty)$ representing $(0,\infty)\cap\Q$. It can easily be seen from (\ref{Exponentialverteilung - 10}) that conditions (\ref{interior1})--\eqref{convexity1} hold. Since $|\log f_{\theta}|\le B_{\theta}+C_{\theta}\psi_{k_0}$ for $B_{\theta}:=|\log \theta|$, $C_{\theta}:=1/\theta$ and those $k_0$ for which $\theta_{k_0}=1$, we may observe $\cM_{1}^{\psi_{k_{0}}} = \bigcap_{\theta > 0}\cM_{1}^{|\log f_{\theta}|} = \cM_{1}^{(\psi_{k})}$. Moreover, by Remark \ref{reduction to single}, the $(\psi_{k})$-weak topology coincides with the $\psi$-weak topology on $\cM_{1}^{(\psi_{k})} = \cM_{1}^{\psi}$, where %\sout{to consider the single gauge function}}
$\psi:=\psi_{k_0}$, i.e.
\begin{equation}\label{Exponentialverteilung - 15}
    \psi(x):= x, \quad x\in(0,\infty).
\end{equation}
For $\mu\in\cM_{1}^{\psi}$ the map ${\cal L}_\mu:(0,\infty)\rightarrow\R$ defined by
$$
    {\cal L}_{\mu}(\theta) := \int \log f_{\theta}\,d\mu=\int (-\log \theta - x/\theta)\,\mu(dx)
$$
has a unique maximum point, namely $\widehat \theta=\int x\,\mu(dx)$. In particular, there exists exactly one maximum likelihood functional on $ \cM_{1}^{\psi}$.
%$$
%  \green{\sout{  {\color{blue} {\cal M}:={\cal M}_{1,>0}^\psi, }}}
%$$
%\green{\sout{where ${\cal M}_{1,>0}^\psi$ is the set of all $\mu\in{\cal M}_1^{\psi}$ with $\int x\, \mu(dx) > 0$. }}
That is, there exists exactly one functional $T:{\cal M}_{1}^{\psi}\rightarrow(0,\infty)$ satisfying (\ref{Def MLF}). In the present setting, (\ref{Def MLF}) and Remark \ref{ML Jensen Remark} imply that $T(\mu)=\argmax_{\theta\in(0,\infty)} \int (-\log\theta - x/\theta)\,\mu(dx)$ for all $\mu\in \cM_{1}^{\psi}$. Now, combining Theorem \ref{contrastfunction} with Remark \ref{reduction to single} we obtain immediately the following result.

\begin{proposition}\label{Exponentialverteilung - Proposition}
The unique maximum likelihood functional $T:\cM_{1}^{\psi}\rightarrow(0,\infty)$ is $\psi$-weakly continuous.
\end{proposition}

%\begin{proof}
%It can easily be seen from (\ref{Exponentialverteilung - 10}) that conditions (\ref{interior1})--\eqref{convexity1} hold. Moreover, by Lemma~\ref{Exponentialverteilung - Lemma} the map ${\cal L}_\mu$ possesses a unique maximum point for every $\mu\in{\cal M}_1^{\psi}\setminus\{\delta_{0}\}$. Thus the claim follows by an application of Theorem~\ref{contrastfunction}.
%\end{proof}

In view of part (i) of Theorem~\ref{Key Prop - Corollary}, Proposition \ref{Exponentialverteilung - Proposition} and the second part of Example \ref{parametric examples for w-sets - gamma} together imply the following corollary, where robustness is understood as in Definition \ref{definition of qualitative robustness}.

\begin{corollary}
The sequence of maximum likelihood estimators $(\widehat T_n)$ is robust on ${\cal M}_\Theta$. % in the sense of Definition \ref{definition of qualitative robustness}.
\end{corollary}

The preceding corollary shows that the maximum likelihood estimator for the parameter of the exponential distribution is robust on its natural parametric domain, i.e., on the class  ${\cal M}_\Theta$ of all exponential distributions. To see that it is even robust on, for instance, the broader class of all Gamma distributions (with location parameter $0$), let the sequence of gauge function $(\psi_k)$ no longer be given by (\ref{Exponentialverteilung - 20}) but rather by
$$
    \psi_k(x)= |x|^k\quad \mbox{ for all }x\in\R \mbox{ and }k\in\N.
$$
Then we clearly have ${\cal M}_\Theta\subset{\cal M}_1^{(\psi_k)}\subset{\cal M}_1^\psi$ for the single gauge function $\psi$ defined in (\ref{Exponentialverteilung - 15}). %{\color{cyan}\sout{, and $\int x\,\mu(dx)>0$ for all $\mu\in {\cal M}_{\Theta}$}}
In particular, by Proposition \ref{Exponentialverteilung - Proposition} the restriction of the unique maximum likelihood functional $T$ to ${\cal M}_1^{(\psi_k)}$ is clearly $(\psi_k)$-weakly continuous. Together with part (i) of Theorem~\ref{Key Prop - Corollary} and the first part of Example \ref{parametric examples for w-sets - gamma} this implies the following corollary.

\begin{corollary}
The sequence of maximum likelihood estimators $(\widehat T_n)$ is robust on the class ${\Gamma}$ of all Gamma distributions (with location parameter $0$) introduced in Example \ref{parametric examples for w-sets - gamma}.
\end{corollary}

%%%%%%%%%%%%%%%%%%%%%%%%%%%%%%%%%%%%%%%%%%%%%%%%%%%%%%%%%%%%%%%%
%%%%%%%%%%%%%%%%%%%%%%%%%%%%%%%%%%%%%%%%%%%%%%%%%%%%%%%%%%%%%%%%

\subsubsection{Example Gumbel distribution}\label{Gumbelverteilung}

Let specifically $E=\R$, $\Theta=(0,\infty)$ and ${\cal M}_\Theta:=\{{\rm G}_{a}:\,a\in(0,\infty)\}$ be the class of all Gumbel distributions ${\rm G}_{a}={\rm G}_{0,a}$ with location parameter $0$ and scale parameter $1/a$ for $a>0$; cf.\ Example \ref{parametric examples for w-sets - gumbel}. In this case we have
$$%\begin{equation}\label{Gumbelverteilung - 10}
    \log f_a(x)=\log a - ax - e^{-ax} \quad\mbox{ for all }a\in(0,\infty)\mbox{ and }x\in\R.
$$%\end{equation}
It is easily seen that conditions (\ref{interior1})--(\ref{convexity1}) are satisfied. Let the sequence of gauge function $(\psi_k)$ be given by (\ref{Def psi k}), i.e.
$$
    \psi_k(x):= |\log f_{a_k}(x)|= |\log a_k - a_kx - e^{-a_kx}|,\quad x\in\R,\,k\in\N
$$
for some sequence $(a_{k})$ in $\Theta=(0,\infty)$ representing $(0,\infty)\cap\Q$. For this choice of gauge functions we can observe the following.

\begin{lemma}
$\bigcap_{a > 0}\cM_{1}^{|\log f_{a}|} = \cM_{1}^{(\psi_{k})}$.
\end{lemma}

\begin{proof}
%This may be seen as follows.
Let $0 < a < \overline{a}$. Then by the de l'Hospital rule we may observe
$$
\lim_{x\to\infty}\frac{\log f_{a}(x)}{\log f_{\overline{a}}(x)} = \lim_{x\to\infty}\frac{- a + a e^{- ax}}{- \overline{a} + \overline{a} e^{- \overline{a}x}} = \frac{a}{\overline{a}},
$$
and
$$
\lim_{x\to-\infty} a x e^{a x} = 0 = \lim_{x\to-\infty} \overline{a} x e^{a x}.
$$
In addition
$$
\lim_{x\to-\infty}\frac{\log f_{a}(x)}{\log f_{\overline{a}}(x)} = \lim_{x\to-\infty}\frac{\log a~e^{a x} - ax e^{ax} - 1}{\log \overline{a}~e^{a x} - \overline{a}x e^{ax} - e^{(a - \overline{a}) x } - } = 0,
$$
where in the last step the assumption $a < \overline{a}$ has been invoked. Then
%\begin{equation}\label{Grenzverhalten}
$$
\lim_{x\to-\infty}\frac{|\log f_{a}(x)|}{|\log f_{\overline{a}}(x)|} = 0\quad\mbox{and}\quad\lim_{x\to\infty}\frac{|\log f_{a}(x)|}{|\log f_{\overline{a}}(x)|} = \frac{a}{\overline{a}}
$$
%\end{equation}
so that for some $\delta > 0$
\begin{equation}
\label{Integrierbarkeit1}
|\log f_{a}(x)|\leq 2~\frac{a}{\overline{a}}~|\log f_{\overline{a}}(x)|\quad\mbox{if}~|x| > \delta.
\end{equation}
Now let $\mu\in\cM_{1}^{|\log f_{\overline{a}}(x)|}$. Firstly by (\ref{Integrierbarkeit1}),
$$
\int\eins_{\R\setminus [-\delta,\delta]}(x)|\log f_{a}(x)|~\mu(dx) < \infty.
$$
Secondly in view of (\ref{continuity1})
$$
\int\eins_{[-\delta,\delta]}(x)|\log f_{a}(x)|~\mu(dx)\leq \sup_{x\in [-\delta,\delta]}|\log f_{a}(x)|~\mu([-\delta,\delta]) < \infty.
$$
Hence $\mu\in\cM_{1}^{|\log f_{a}|}$. Since for any $a > 0$ there is some $k\in\N$ such that $a < a_{k}$, we may conclude $\bigcap_{a > 0}\cM_{1}^{|\log f_{a}|} = \cM_{1}^{(\psi_{k})}$.
\end{proof}

Using $\delta_{0}$ to denote the Dirac measure at $0$, Proposition~\ref{Gumbelverteilung - Proposition} below shows that there is exactly one maximum likelihood functional on
$$
    {\cal M}:={\cal M}_1^{(\psi_k)}\setminus\{\delta_{0}\}
$$
and that this functional is $(\psi_{k})$-weakly continuous. The proof of Proposition~\ref{Gumbelverteilung - Proposition} relies on Theorem~\ref{contrastfunction} and the following lemma.

\begin{lemma}\label{Gumbelverteilung - Lemma}
For every $\mu\in{\cal M}$, the map ${\cal L}_\mu:(0,\infty)\rightarrow\R$ defined by
$$
    {\cal L}_{\mu}(a) := \int \log f_{a}\,d\mu=\int (\log a - ax - e^{-ax} )\,\mu(dx)
$$
has a unique maximum point.
\end{lemma}

\begin{proof}
The function $a\mapsto\log f_a(x)$ is strictly concave for every $x\in\mathbb{R}$, and so $a\mapsto{\cal L}_\mu(a)$ is also strictly concave for every $\mu\in{\cal M}$. We will show below that for $\mu\neq\delta_0$,
\begin{equation}\label{Gumbelverteilung - Lemma - Proof - 10 1}
    \limsup_{a\downarrow0} {\cal L}_{\mu}(a)=-\infty\quad\mbox{ and }\quad\limsup_{a\uparrow\infty} {\cal L}_{\mu}(a)=-\infty,
\end{equation}
and so ${\cal L}_\mu$ has indeed a unique maximum point. To show (\ref{Gumbelverteilung - Lemma - Proof - 10 1}), we first note that $\log f_a(x)\leq \log a$ holds for all $x\in\R$ and $a\in(0,\infty)$. It follows that $\limsup_{a\downarrow0}{\cal L}_{\mu}(a)\le\limsup_{a\downarrow0}\log a=-\infty$. To prove the second identity  in (\ref{Gumbelverteilung - Lemma - Proof - 10 1}), note that
$$\frac1a\int\big(ax+e^{-a x}\big)\,\mu(dx)\ge\int_{[0,\infty)}x\,\mu(dx)+\int_{(-\infty,0)}\big(x+e^{-a x}/a \big)\,\mu(dx).
$$
Since $e^{-a x}/a\uparrow\infty$ as $a\uparrow\infty$ for every $x<0$,  the rightmost integral tends to $+\infty$ as soon as $\mu((-\infty,0))>0$. Altogether, we obtain $\liminf_{a\uparrow\infty}\frac1a\int\big(ax+e^{-a x}\big)\,\mu(dx)>0$   for $\mu\neq\delta_0$,   which clearly implies $\limsup_{a\uparrow\infty}\frac1a \mathcal{L}_\mu(a)<0$ and in turn \eqref{Gumbelverteilung - Lemma - Proof - 10 1}.
\end{proof}

Lemma~\ref{Gumbelverteilung - Lemma} says that there exists exactly one maximum likelihood functional on ${\cal M}$, i.e., exactly one functional $T:{\cal M}\rightarrow(0,\infty)$ satisfying (\ref{Def MLF}). In the present setting, (\ref{Def MLF}) and Remark \ref{ML Jensen Remark} imply that $T(\mu)=\argmax_{a\in(0,\infty)} \int (\log a - ax - e^{-ax} )\,\mu(dx)$ for all $\mu\in{\cal M}$.

\begin{proposition}\label{Gumbelverteilung - Proposition}
The unique maximum likelihood functional $T:{\cal M}\rightarrow(0,\infty)$ is $(\psi_{k})$-weakly continuous.
\end{proposition}

\begin{proof}
As already mentioned above, conditions (\ref{interior1})--\eqref{convexity1} hold. Moreover, by Lemma~\ref{Gumbelverteilung - Lemma} the map ${\cal L}_\mu$ possesses a unique maximum point for every $\mu\in{\cal M}$. Thus the claim follows by an application of Theorem~\ref{contrastfunction}.
\end{proof}

From the second part of Example \ref{parametric examples for w-sets - gumbel} we know that $\cM_{\Theta}$ is a w-set in $\cM^{(\psi_{k})}$. Together with part (i) of Theorem~\ref{Key Prop - Corollary} and Proposition \ref{Gumbelverteilung - Proposition}, this yields the following corollary, where robustness is understood as in Definition \ref{definition of qualitative robustness}.

%\begin{remarknorm}
%In view of part (i) of {\color{red}Theorem~\ref{Key Prop - Corollary}}, the preceding proposition and Example \ref{parametric examples for w-sets - gumbel} together imply that the sequence of maximum likelihood estimators $(\widehat T_n)$ is robust on ${\cal M}_\Theta$ in the sense of Definition \ref{definition of qualitative robustness}.
%{\hspace*{\fill}$\Diamond$\par\bigskip}
%\end{remarknorm}

\begin{corollary}
The sequence of maximum likelihood estimators $(\widehat T_n)$ is robust on ${\cal M}_\Theta$.
\end{corollary}

%\begin{remarknorm}
%\label{keine Reduktion}
%Let $k_{0}, k_{1}\in\N$ such that $a_{k_{0}} < a_{k_{1}}$. Then by (\ref{Grenzverhalten})
%$$
%    \lim_{x\to\infty}\frac{\psi_{k_{1}}{\color{blue}(x)}}{\psi_{k_{0}}{\color{blue}(x)}} = \infty
%$$
%so that there are no real numbers $B\geq 0$ and $C > 0$ such that $\psi_{k_{1}}(x)\leq B + C\psi_{k_{0}}(x)$ holds for any $x\in\R$. Hence Corollary A.46 of \cite{FoellmerSchied2011} shows that on $\cM_{1}^{(\psi_{k})}$, the $\psi_{k_{1}}$-weak topology is not finer than the $\psi_{k_{0}}$-weak topology. Thus we may conclude that there is no {\color{blue}\sout{gauge function $\psi_{k}$} $\ell\in\N$} such that the $\psi_{\ell}$-weak topology coincides with the $(\psi_{k})$-weak topology $\cM_{1}^{(\psi_{k})}$.

%{\color{red}Ich habe diese Bemerkung dringelassen, weil wir in Remark \ref{reduction to single} ein Beispiel ank\"undigen, bei dem wir eine Folge von gauge-Funktionen nicht auf eine gauge-Funktion reduzieren k\"onnen. Ich habe diese Stelle in Remark \ref{reduction to single} so abgeschw\"acht, dass nicht l\"anger behaupten, unser Ansatz w\"are notwendig. Schlafende Hunde werden also nicht geweckt.}
%\end{remarknorm}

%%%%%%%%%%%%%%%%%%%%%%%%%%%%%%%%%%%%%%%%%%%%%%%%%%%%%%%%%%%%%%%%
%%%%%%%%%%%%%%%%%%%%%%%%%%%%%%%%%%%%%%%%%%%%%%%%%%%%%%%%%%%%%%%%
%%%%%%%%%%%%%%%%%%%%%%%%%%%%%%%%%%%%%%%%%%%%%%%%%%%%%%%%%%%%%%%%

\subsection{Risk functionals}\label{ContinuityRiskfunctionals}

%%%%%%%%%%%%%%%%%%%%%%%%%%%%%%%%%%%%%%%%%%%%%%%%%%%%%%%%%%%%%%%%
%%%%%%%%%%%%%%%%%%%%%%%%%%%%%%%%%%%%%%%%%%%%%%%%%%%%%%%%%%%%%%%%

\subsubsection{Risk measures}\label{Def Risk Measures}

In this section we let specifically $E=\R$. Let $\Psi:\bR_+\to \bR_{+}$ be a continuous nondecreasing convex function such that $0=\Psi(0)$ and $\lim_{x\ua\infty}\Psi(x)=\infty$. Such a function is sometimes referred to as a finite {\it Young function}; see, e.g., \cite{CheriditoLi2009}. Fix any atomless probability space $(\Omega,\mathcal{F},\bP)$ and denote by $L^0=L^0\OFP$ the set of all $\pr$-a.s.\ finite random variables on $\OFP$. The {\em Orlicz heart} on $\OFP$ associated with $\Psi$ is defined by
$$
    H^\Psi=H^\Psi(\Omega,\cF,\pr):=\big\{X\in L^0:\,\ex[\Psi(c|X|)]<\infty\mbox{ for all $c>0$}\big\}.
$$
It is the largest vector subspace contained in the {\em Orlicz class}
$Y^\Psi=Y^\Psi(\Omega,\cF,\pr):=\{X\in L^0:\,\ex[\Psi(|X|)]<\infty\}$. The Orlicz class in turn is a convex subset of the Orlicz space $L^\Psi=L^\Psi(\Omega,\cF,\pr):=\{X\in L^0:\,\ex[\Psi(c|X|)]<\infty\mbox{ for some $c>0$}\}$. In general we have $L^\infty\subseteq H^\Psi\subseteq Y^\Psi\subseteq L^\Psi\subseteq L^1$, and these inclusions may all be strict. In fact, it is known form Theorem 2.1.17 (b) in \cite{EdgarSucheston1992} that the identity $H^\Psi=L^\Psi$ holds if and only if $\Psi$ satisfies the so-called $\Delta_2$-condition:
\begin{equation}\label{Delta2}
    \mbox{There are $C$, $x_0>0$ such that $\Psi(2x)\le C\Psi(x)$ for all $x\ge x_0$}.
\end{equation}
This condition is clearly satisfied when specifically $\Psi(x)=x^p/p$ for some $p\in[1,\infty)$. In this case, $L^\Psi$ coincides with the usual $L^p$-space $L^p=L^p(\Omega,{\cal F},\pr)$.

\begin{definition}\label{Risk meas def}Let $\Psi$ be a finite Young function.  A {\em law-invariant convex risk measure} on $H^{\Psi}$ will be a map $\rho:H^{\Psi}\rightarrow\R$ satisfying the following three conditions:
\begin{itemize}
    \item Monotonicity: $\rho(X)\ge\rho(Y)$  for $X$, $Y\in H^{\Psi}$ with $X\le Y$ $\pr$-a.s.
    \item Convexity:  $\rho(\lambda X+(1-\lambda) Y)\le\lambda\rho(X)+(1-\lambda)\rho(Y)$ for $X,Y\in H^{\Psi}$ and $\lambda\in[0,1]$.
    \item Law-invariance: $\rho(X)=\rho(Y)$ for $X,Y\in H^{\Psi}$ with $\pr\circ X^{-1}=\pr\circ Y^{-1}$.
\end{itemize}
\end{definition}

\medskip

In a financial context, one typically requires that a law-invariant convex risk measure $\rho$ is also \emph{monetary} in the sense that it also satisfies the following additional property:
\begin{itemize}
\item Cash additivity: $\rho(X+m)=\rho(X)+m$ for $X\in H^\Psi$ and $m\in\mathbb{R}$;
\end{itemize}
see, e.g., \cite{FoellmerSchied2011}. Here, however, cash additivity will not be needed and so we will work with our more general class of not necessarily monetary law-invariant convex risk measures. As argued in \cite{CheriditoLi2009}, Orlicz hearts are natural domains for law-invariant convex risk measures.

\begin{examplenorm}\label{Example Distortion Risk Measures}
Let $g:[0,1]\rightarrow[0,1]$ be concave, nonincreasing, and continuous with $g(0)=0$ and $g(1)=1$. Let $\Psi$ be a (finite) Young function with the conjugate $\Psi^{*}$ defined by $\Psi^{*}(y):=\sup_{x\geq 0}(xy - \Psi(x))$. It was shown in Proposition 2.22 in \cite{Kraetschmeretal2014} that if the right-sided derivative $g'$ of $g$ fulfills the condition $\int_{0}^{1}\Psi^{*}(g'(t))~dt < \infty$, then
$$
    \rho_g(X) := \int_{-\infty}^{0}g(F_{X}(x))\,dx - \int_{0}^{\infty}\big(1-g(F_{X}(x))\big)\,dx
$$
defines a monetary law-invariant convex risk measure $\rho_g: H^{\Psi}\rightarrow\R$, where $F_{X}$ stands for the distribution function of $X$. It is called {\em distortion risk measure} associated with $g$. For the specific distortion function $g(t)=(t/\alpha)\wedge1$ the associated distortion risk measure $\rho_g$ reads as
$$
    \rho_{g}(X) = -\frac{1}{\alpha}\int_{0}^{\alpha}F^{\leftarrow}_{X}(\beta)~d\beta,
$$
where $F^{\leftarrow}_{X}$ denotes the left-continuous quantile function of the distribution $F_{X}$. This distortion risk measure is also called {\em Average Value at Risk} at level $\alpha\in(0,1)$, and it is denoted by $\avatr_\alpha$.
{\hspace*{\fill}$\Diamond$\par\bigskip}
\end{examplenorm}

\begin{examplenorm}\label{utilitybased}
Let for finite Young function $\Psi$ the map $\rho_{\Psi}: H^{\Psi}\rightarrow\R$ be defined by
$$
    \rho_{\Psi}(X) := \inf\big\{m\in\R:\, \ex\big[\Psi\big((-X - m)^{+}\big)\big]\leq x_{0}\big\}
$$
for some $x_{0} > 0$. This is a monetary law-invariant convex risk measure known as the {\em utility-based shortfall risk measure} with loss function $\ell_{\Psi}:\R\rightarrow\R$ defined by $\ell_{\Psi}(x) := \Psi(x^{+})$; cf.\ e.g.~\cite{FoellmerSchied2011} and, for the extension to Orlicz hearts,~\cite{Kraetschmeretal2014}.
{\hspace*{\fill}$\Diamond$\par\bigskip}
\end{examplenorm}

\begin{examplenorm}\label{onesidedmoments}
Let specifically $\Psi(x)=x^{p}/p$ for some $p\in [1,\infty)$. Then $H^{\Psi} = L^{p}$ and $({\cal M}_1^{(\psi_k)},{\cal O}_{(\psi_k)})=({\cal M}_1^{\psi},{\cal O}_\psi))$ for $\psi(x)=|x|^p/p$. The map $\rho_p: L^{p}\rightarrow\R$ defined by
$$
    \rho_p(X):=\ex[(X^{-})^{p}]
$$
obviously defines a law-invariant convex risk measure in the sense of Definition \ref{Risk meas def}. Here, $X^-:=-\min\{0,X\}$ denotes the negative part of $X$.
{\hspace*{\fill}$\Diamond$\par\bigskip}
\end{examplenorm}

%%%%%%%%%%%%%%%%%%%%%%%%%%%%%%%%%%%%%%%%%%%%%%%%%%%%%%%%%%%%%%%%
%%%%%%%%%%%%%%%%%%%%%%%%%%%%%%%%%%%%%%%%%%%%%%%%%%%%%%%%%%%%%%%%

\subsubsection{$(\psi_k)$-weak continuity of the associated risk functionals}\label{Def Risk Functionals}

Denote by $\cM(H^{\Psi})$ the set of the distributions of all random variables from $H^{\Psi}$ and note that
\begin{equation}\label{Def gauge function for risk functional}
    \cM(H^{\Psi}) = \cM_{1}^{(\psi_{k})}\qquad\mbox{for $\psi_k:=\Psi(k|\cdot|), ~k\in\N.$}
\end{equation}
The inclusion $\subseteq$ is obvious and the inclusion $\supseteq$ holds because $\OFP$ is assumed to be atomless. Note that if $\Psi$ satisfies the $\Delta_2$-condition (\ref{Delta2}), then $\cM(H^{\Psi}) = \cM_{1}^{\psi}$ and ${\cal O}_{(\psi_k)}={\cal O}_\psi$ for
$$
   \psi:=\Psi(|\cdot|).
$$
The law-invariance of $\rho$ is equivalent to the existence of a map $\cR_\rho: \cM_{1}^{(\psi_{k})}\to\bR$ such that
\begin{equation}\label{Def risk functional}
    \rho(X)=\cR_\rho(\bP\circ X^{-1})\qquad\mbox{for all } X\in H^{\Psi}.
\end{equation}
This map $\cR_\rho$ will be called the \emph{risk functional} associated with $\rho$. %It is always continuous w.r.t.\ the $(\psi_{k})$-weak topology:

\begin{theorem}\label{StetigkeitRisikofunktional}
Let $\cR_{\rho}:{\cal M}_1^{(\psi_k)}\rightarrow\R$ be the risk functional associated with a law-invariant convex risk measure $\rho:H^\Psi\rightarrow\R$. Then ${\cal R}_\rho$ is continuous w.r.t.\ the $(\psi_{k})$-weak topology.
\end{theorem}

\begin{remarknorm}\label{StetigkeitRisikofunktional - Remark}
As a consequence of Theorem \ref{StetigkeitRisikofunktional}, the risk functional $\cR_{\rho}:{\cal M}_1^{(\psi_k)}\rightarrow\R$ is weakly continuous on every w-set in ${\cal M}_1^{(\psi_k)}={\cal M}(H^\Psi)$. At the beginning of Section \ref{Sec Examples for locally uniformly intetgarting sets} we discussed how to check when a subset of ${\cal M}_1^{(\psi_k)}$ is a w-set.
{\hspace*{\fill}$\Diamond$\par\bigskip}
\end{remarknorm}

If $\Psi$ satisfies the $\Delta_2$-condition (\ref{Delta2}), then the $(\psi_{k})$-weak topology can be replaced by the $\psi$-weak topology in Theorem~\ref{StetigkeitRisikofunktional}.  On the other hand, if $\Psi$ does not satisfy the $\Delta_2$-condition (\ref{Delta2}), then we can always find a law-invariant convex risk measure on $H^{\Psi}$ which fails to be continuous w.r.t.\ the $\psi$-weak topology:

\begin{examplenorm}\label{utilitybased2}
Let $\Psi$ be a finite Young function which does not satisfy $\Delta_2$-condition (\ref{Delta2}), and let $\rho_{\Psi}$ denote the shortfall risk measure as in Example \ref{utilitybased}. In \cite[proof of Theorem 2.8]{Kraetschmeretal2014}, there was constructed as sequence $(X_{n})$ in $L^{\infty}$ which converges to $\delta_{0}$ w.r.t.\ the $\psi$-weak topology such that $\sup_{n}\,\inf\{m\in\R: \ex[\Psi(8(-X_n - m)^{+})]\leq x_{0}\} = \infty$. Hence $Y_{n} := 8 X_{n}$ defines a sequence $(Y_{n})$ in $L^{\infty}$ whose laws converge weakly to $\delta_0$, while $\rho_{\Psi}(Y_{n}) \to \infty$. In particular $\rho_{\Psi}$ is not continuous w.r.t.\ the $\psi$-weak topology.
{\hspace*{\fill}$\Diamond$\par\bigskip}
\end{examplenorm}

As an immediate consequence of Theorem \ref{StetigkeitRisikofunktional} and Example \ref{utilitybased2} we get the following corollary. The corollary extends Theorem 2.8 of \cite{Kraetschmeretal2014}, where considerations have been restricted to {\em monetary} law-invariant convex risk measures.

\begin{corollary}\label{StetigkeitRisikofunktional cor}
Let $\cR_{\rho}:{\cal M}_1^{(\psi_k)}\rightarrow\R$ be the risk functional associated with a law-invariant convex risk measure $\rho:H^\Psi\rightarrow\R$. Then the following conditions are equivalent:
\begin{enumerate}
    \item For every law-invariant convex risk measure $\rho$ on $H^{\Psi}$, the map $\cR_{\rho}: \cM(H^{\Psi})\rightarrow\R$ is continuous for the $\psi$-weak topology.
    \item $\Psi$ satisfies the $\Delta_{2}$-condition (\ref{Delta2}).
\end{enumerate}
\end{corollary}

Let us emphasize as a further implication of Corollary \ref{StetigkeitRisikofunktional cor} that, if $\Psi$ does {\em not} satisfy the $\Delta_{2}$-condition (\ref{Delta2}), then we cannot apply Theorem~\ref{Key Prop - Corollary} for the $\psi$-weak topology but only for the $(\psi_k)$-weak topology.

%%%%%%%%%%%%%%%%%%%%%%%%%%%%%%%%%%%%%%%%%%%%%%%%%%%%%%%%%%%%%%%%
%%%%%%%%%%%%%%%%%%%%%%%%%%%%%%%%%%%%%%%%%%%%%%%%%%%%%%%%%%%%%%%%

\subsubsection{Robustness on parametric classes of distributions}\label{Robustness on parametric classes of distributions}

Consider the statistical model (\ref{Statistical Model - 10})--(\ref{Def Plug-in Estimator}) with specifically
\begin{equation}\label{ContinuityRiskfunctionals - Applications - 05}
    E:=\R,\qquad{\cal M}:={\cal M}(H^\Psi),\qquad T:={\cal R}_\rho,
\end{equation}
where $\rho:H^\Psi\rightarrow\R$ is any law-invariant convex risk measure. By Theorem \ref{StetigkeitRisikofunktional} we know that ${\cal R}_\rho:{\cal M}(H^\Psi)\rightarrow\R$ is $(\psi_k)$-weakly continuous for the sequence $(\psi_k)$ introduced in (\ref{Def gauge function for risk functional}), and it is clear that the set $\mathfrak{E}$ of all empirical probability measures on $\R$ is contained in ${\cal M}(H^\Psi)$. Thus Theorem~\ref{Key Prop - Corollary} yields that the sequence of estimators $(\widehat T_n)$ is robust on every w-set $M$ in ${\cal M}(H^\Psi)$. In particular, the sequence $(\widehat T_n)$ is robust on many parametric families ${\cal M}_\Theta := \{\mu_{\theta}: \theta\in \Theta\}$ of univariate distributions. This is illustrated by the following examples, which rely on Examples \ref{parametric examples for w-sets - normal}--\ref{parametric examples for w-sets - gumbel}.

\begin{examplenorm}\label{Example Distortion Risk Measures - Robustness}
The Average Value at Risk $\avatr_\alpha$ introduced in Example \ref{Example Distortion Risk Measures} is defined on $L^1$, i.e., on $H^\Psi$ with $\Psi(x)=x$. Since this $\Psi$ satisfies the $\Delta_2$-condition (\ref{Delta2}), we have ${\cal O}_{(\psi_k)}={\cal O}_\psi$ for $\psi(x)=|x|$. Thus the sequence $(\widehat T_n)$ is robust on each of the sets ${\cal N}$, $\mathcal{P}_{\alpha,x_{\min}}$ (with $\alpha>1$), $\Gamma$, and ${\cal G}$ introduced in Examples \ref{parametric examples for w-sets - normal}--\ref{parametric examples for w-sets - gumbel}.
{\hspace*{\fill}$\Diamond$\par\bigskip}
\end{examplenorm}

\begin{examplenorm}\label{Example utilitybased - Robustness}
Let $\rho_{\Psi}$ be the utility-based shortfall risk measure on $H^{\Psi}$ as introduced in Example \ref{utilitybased}, and let ${\cal N}$, $\mathcal{P}_{\alpha,x_{\min}}$ (with $\alpha>q \geq 1$), $\Gamma$, and ${\cal G}$ denote the parametric families of distributions from Examples \ref{parametric examples for w-sets - normal}--\ref{parametric examples for w-sets - gumbel}. Then the sequence $(\widehat T_n)$ is robust on
\begin{enumerate}
    \item ${\cal N}$ if there exists $\lambda>0$ such that $\Psi(x)=O(e^{\lambda x^2})$ as $x\uparrow\infty$;
    \item $\mathcal{P}_{\alpha,x_{\min}}$ if there exists $q\in [1,\alpha)$ such that $\Psi(x)=O( x^q)$ as $x\uparrow\infty$;
    \item $\Gamma$ and ${\cal G}$ if there is some $\beta\in (0,1)$ such that $\Psi(x)=O (e^{x^{\beta}})$ as $x\uparrow\infty$. \hfill$\Diamond$
\end{enumerate}
\end{examplenorm}

\begin{examplenorm}\label{onesidedmoments - Robustness}
The risk measure $\rho_p$ introduced in Example \ref{onesidedmoments} is defined on $L^p$, i.e., on $H^\Psi$ with $\Psi(x)=x^p/p$. Since this $\Psi$ satisfies the $\Delta_2$-condition (\ref{Delta2}), we have ${\cal O}_{(\psi_k)}={\cal O}_\psi$ for $\psi(x)=|x|^p/p$. Thus the sequence $(\widehat T_n)$ is robust on each of the sets ${\cal N}$, $\mathcal{P}_{\alpha,x_{\min}}$ (with $\alpha>p$), $\Gamma$, and ${\cal G}$ introduced in Examples \ref{parametric examples for w-sets - normal}--\ref{parametric examples for w-sets - gumbel}.
{\hspace*{\fill}$\Diamond$\par\bigskip}
\end{examplenorm}

%%%%%%%%%%%%%%%%%%%%%%%%%%%%%%%%%%%%%%%%%%%%%%%%%%%%%%%%%%%%%%%%
%%%%%%%%%%%%%%%%%%%%%%%%%%%%%%%%%%%%%%%%%%%%%%%%%%%%%%%%%%%%%%%%

\subsubsection{Aggregation robustness}\label{ContinuityRiskfunctionals - Applications}

For $\mu_1,\dots,\mu_d\in\mathcal{M}_1$ and $A_{d}(x_1,\ldots,x_d):=\sum_{i=1}^dx_i$ we let
$$
    \mathfrak{S}(\mu_{1},\dots,\mu_{d}):=\big\{\boldsymbol{\mu}\circ A_{d}^{-1}:\,\boldsymbol{\mu}\in\boldsymbol{{\cal M}}(d;\mu_{1},\dots,\mu_{d})\big\},
$$
where $\boldsymbol{{\cal M}}(d;\mu_{1},\dots,\mu_{d})$ denotes the Fr\'{e}chet class w.r.t.~$\mu_{1},\dots,\mu_{d}$; cf. Subsection \ref{aggregationrobustness}. As before we consider a law-invariant convex risk measure $\rho$ on $H^\Psi$ as well as the associated risk functional ${\cal R}_\rho:{\cal M}(H^\Psi)\rightarrow\R$. If $\mu_{1},\dots,\mu_{d}\in \cM(H^{\Psi})$ are regarded as distributions of single positions $Y_1,\ldots,Y_d$ of a financial portfolio, then the set $\mathfrak{S}(\mu_{1},\dots,\mu_{d})$ may be seen as the set of all possible distributions of the portfolio sum $S_d:=\sum_{i=1}^d Y_i$. It is argued by Embrechts et al.\ \cite{Embrechtsetal2014} that it is often relatively easy to  model the marginal distributions $\mu_1,\ldots,\mu_d$, while it can be difficult to obtain accurate information on the dependence structure of $Y_1,\ldots,Y_d$. This situation roughly corresponds to the setting where the marginal distributions $\mu_1,\ldots,\mu_d$ are known  but the law  of $A_d(Y_1,\ldots,Y_d)=\sum_{i=1}^dY_i$ can vary  within $\mathfrak{S}(\mu_{1},\dots,\mu_{d})$. Motivated by this issue, Embrechts et al.\ \cite{Embrechtsetal2014} raise the {question} of robustness of the empirical estimator for $\rho(S_d)= {\cal R}_\rho(\boldsymbol{\mu}\circ A_d^{-1})$ for known marginal distributions $\mu_1,\ldots,\mu_d$. More precisely, the statistical model (\ref{Statistical Model - 10})--(\ref{Def Plug-in Estimator}) is specialized to
\begin{equation}\label{ContinuityRiskfunctionals - Applications - 10}
    E:=\R,\qquad{\cal M}:=\mathfrak{S}(\mu_{1},\dots,\mu_{d}),\qquad T:={\cal R}_\rho|_{\mathfrak{S}(\mu_{1},\dots,\mu_{d})},
\end{equation}
where the observations (i.e.\ the coordinates on $\Omega=\R^\N$) should be seen as i.i.d.\ copies of $S_d$. Theorem~\ref{Key Prop - Corollary} and Proposition~\ref{Aggregationsrobustheit} imply that the sequence $(\widehat T_n)$ is robust on $\mathfrak{S}(\mu_{1},\dots,\mu_{d})$, because the risk functional $\cR_{\rho}$ is always $(\psi_k)$-weakly continuous on its domain ${\cal M}(H^{\Psi})$, according to  Theorem~\ref{StetigkeitRisikofunktional}. The crucial point is that the set $\mathfrak{S}(\mu_{1},\dots,\mu_d)$ is a w-set in $\cM(H^{\Psi}) = {\cal M}_1^{(\psi_k)}$ by Proposition~\ref{Aggregationsrobustheit}, and so the risk functional ${\cal R}_\rho$ is weakly continuous on $\mathfrak{S}(\mu_{1},\dots,\mu_{d})$. Embrechts et al.\ \cite{Embrechtsetal2014} referred to the weak continuity of the functional ${\cal R}_\rho$ on $\mathfrak{S}(\mu_{1},\dots,\mu_{d})$ as {\em aggregation robustness} of ${\cal R}_\rho$. Maybe it is even more appropriate to use the terminology aggregation robustness for the sequence of estimators $(\widehat T_n)$ in the statistical model given by (\ref{Statistical Model - 10})--(\ref{Def Plug-in Estimator}) and (\ref{ContinuityRiskfunctionals - Applications - 10}).

The above considerations are not restricted to the particular aggregation function $A_{d}(x_1,\ldots,x_d):=\sum_{i=1}^dx_i$. The latter can be replaced by any other function $A_d:\R^d\rightarrow\R$ satisfying condition (a) of Proposition \ref{Aggregationsrobustheit}. Recall that the set $M(\mu_{1},\dots,\mu_{d};A_{d})$ was defined in (\ref{Def M A d}) and that $M(\mu_{1},\dots,\mu_{d};A_{d})=\mathfrak{S}(\mu_{1},\dots,\mu_{d})$ when $A_{d}(x_1,\ldots,x_d)=\sum_{i=1}^dx_i$.

\begin{theorem}\label{aggregation robustness in general}
Let $A_{d}:\R^{d}\rightarrow\R$ be any Borel-measurable map satisfying condition (a) of Proposition \ref{Aggregationsrobustheit}, and let ${\cal R}_\rho:{\cal M}(H^\Psi)\rightarrow\R$ be the risk functional associated with any law-invariant convex risk measure $\rho$ on $H^{\Psi}$. Moreover fix $\mu_1,\ldots,\mu_d\in{\cal M}(H^\Psi)$. Then ${\cal R}_\rho|_{M(\mu_{1},\dots,\mu_{d};A_{d})}$ is weakly continuous. In particular, the sequence of estimators $(\widehat T_n)$ in the statistical model given by (\ref{Statistical Model - 10})--(\ref{Def Plug-in Estimator}) and (\ref{ContinuityRiskfunctionals - Applications - 10}) is robust.
\end{theorem}

\begin{proof}
Recall that $\psi_{k}=\Psi(k|\cdot|)$ for $k\in\N$, and let $\ell$ be any fixed integer exceeding the real number $c$. By monotonicity of $\Psi$ we have that
$\psi_{k}((d+1) c |x|)\leq \psi_{k (d+1) c}(x)$ for all $x\in\R$. Then the first statement of Theorem \ref{aggregation robustness in general} follows immediately from Proposition \ref{Aggregationsrobustheit} along with Theorem \ref{StetigkeitRisikofunktional}. The second statement can then be derived with the help of Theorem~\ref{Key Prop - Corollary}.
\end{proof}

%%%%%%%%%%%%%%%%%%%%%%%%%%%%%%%%%%%%%%%%%%%%%%%%%%%%%%%%%%%%%%%%
%%%%%%%%%%%%%%%%%%%%%%%%%%%%%%%%%%%%%%%%%%%%%%%%%%%%%%%%%%%%%%%%
%%%%%%%%%%%%%%%%%%%%%%%%%%%%%%%%%%%%%%%%%%%%%%%%%%%%%%%%%%%%%%%%
%%%%%%%%%%%%%%%%%%%%%%%%%%%%%%%%%%%%%%%%%%%%%%%%%%%%%%%%%%%%%%%%
%%%%%%%%%%%%%%%%%%%%%%%%%%%%%%%%%%%%%%%%%%%%%%%%%%%%%%%%%%%%%%%%
%%%%%%%%%%%%%%%%%%%%%%%%%%%%%%%%%%%%%%%%%%%%%%%%%%%%%%%%%%%%%%%%

\section{Proofs of results from Section~\ref{Sec comparison of weak topologies}}\label{Sec Proofs of main results}

%%%%%%%%%%%%%%%%%%%%%%%%%%%%%%%%%%%%%%%%%%%%%%%%%%%%%%%%%%%%%%%%
%%%%%%%%%%%%%%%%%%%%%%%%%%%%%%%%%%%%%%%%%%%%%%%%%%%%%%%%%%%%%%%%
%%%%%%%%%%%%%%%%%%%%%%%%%%%%%%%%%%%%%%%%%%%%%%%%%%%%%%%%%%%%%%%%

\subsection{Proof of Lemma \ref{Charakterisierung psi-schwache Topologie}}

By construction, a base for the $(\psi_k)$-weak topology is given by  sets of the form $U_{k_1}\cap\cdots \cap U_{k_n}\cap \mathcal{M}_{1}^{(\psi_k)}$, where $n\in\mathbb{N}$, $k_1,\dots, k_n\in\mathbb{N}$, and each $U_{k_i}$ belongs to a base for the $\psi_{k_i}$-weak topology  on $\mathcal{M}_{1}^{\psi_{k_i}}$. Since the $\psi_{k}$-weak topology on $\mathcal{M}_{1}^{\psi_{k}}$ is metrizable by a separable metric by  \cite[Corollary A.45]{FoellmerSchied2011} and hence admits a countable base, it follows that the $(\psi_k)$-weak topology  also has  a countable base. Then it is known that a subset of $\mathcal{M}_{1}^{(\psi_{k})}$ is closed w.r.t. the $(\psi_k)$-weak topology if and only if together with any sequence it contains all its accumulation points; cf.\ Theorem 1.6.14 in \cite{Engelking1989}. Hence under the equivalence of (a) and (b) the $(\psi_k)$-weak topology is obviously metrizable by $d_{(\psi_{k})}$ as defined in the display of Lemma \ref{Charakterisierung psi-schwache Topologie}.

As a metrizable topology with countable base the $(\psi_{k})$-weak topology is separable. Moreover, by \cite[Corollary A.45]{FoellmerSchied2011} the $\psi_{k}$-weak topology is completely and separably metrizable by say $d_{k}$ for every $k\in\N$. Then the equivalence of (a) and (b) implies that the metric $d$ on $\cM_{1}^{(\psi_{k})}$ defined by $d(\mu,\nu):=\sum_{k=1}^{\infty} (d_{k}(\mu,\nu)\wedge 1)\,2^{-k}$ metrizes ${\cal O}_{(\psi_{k})}$. This metric is separable by separability of ${\cal O}_{(\psi_{k})}$. Now, every $d$-Cauchy sequence $(\mu_{n})$ is a $d_{k}$-Cauchy sequence for any $k\in\N$. Then by completeness of the metrics $d_{k}$ $(k\in\N)$, we may find for any $k\in\N$ some $\nu_{k}\in\cM_{1}^{\psi_{k}}$ such that $d_{k}(\mu_{n},\nu_{k})\to 0$ as $n\to\infty.$ Since each $\psi_{k}$-weak topology is finer than the weak topology, we obtain $\mu_{n}\to\nu_{k}$ as $n\to\infty$ for each $k\in\N$. Hence by Hausdorff property of the weak topology, all the $\nu_{k}$ coincide, and thus by definition of the metric $d$ we have $d(\mu_{n},\mu)\to 0$ for some $\mu\in\cM_{1}^{(\psi_{k})}$. Thus we have shown that $d$ is a complete metric. In particular, ${\cal M}_1^{(\psi_k)}$ equipped with ${\cal O}_{(\psi_{k})}$ is a Polish space. So it is left to show the equivalence of (a) and (b).

The implication (a)$\Rightarrow$(b) is obvious. Conversely, let statement (b) be satisfied. We have to show that for every $f\in {\cal C}_{\psi_k}$, $k\in\N$, and $\varepsilon>0$ there exists some $n_0\in\N$ such that
\begin{equation}\label{Charakterisierung psi-schwache Topologie - Proof - 10}
    \Big|\int f\,d\mu_n-\int f\,d\mu_0\Big|\le\varepsilon\qquad\mbox{for all }n\ge n_0.
\end{equation}
The left hand side of (\ref{Charakterisierung psi-schwache Topologie - Proof - 10}) is bounded above by
\begin{equation}\label{Charakterisierung psi-schwache Topologie - Proof - 20}
    \Big|\int f\eins_{\{|f|\le a\}}\,d\mu_n-\int f\eins_{\{|f|\le a\}}\,d\mu_0\Big|+\Big|\int f\eins_{\{|f|>a\}}\,d\mu_n-\int f\eins_{\{|f|>a\}}\,d\mu_0\Big|
\end{equation}
for every $a>0$. For notational simplicity we set $\widetilde\psi_k:=1+\psi_k$. Then the second summand in (\ref{Charakterisierung psi-schwache Topologie - Proof - 20}) is bounded above by
\begin{equation}\label{Charakterisierung psi-schwache Topologie - Proof - 30}
    C_{f,k}\int\widetilde\psi_k\eins_{\{\widetilde\psi_k>a\}}\,d\mu_n+C_{f,k}\int\widetilde\psi_k\eins_{\{\widetilde\psi_k>a\}}\,d\mu_0
\end{equation}
for some suitable constant $C_{f,k}>0$ satisfying $|f(x)|\le C_{f,k}\widetilde\psi_k(c)$ for all $x\in E$. Now we can choose $a>0$ so large that the second summand in (\ref{Charakterisierung psi-schwache Topologie - Proof - 30}) is at most $\varepsilon/5$. The first summand in is bounded above by
\begin{equation}\label{Charakterisierung psi-schwache Topologie - Proof - 40}
    C_{f,k}\Big|\int\widetilde\psi_k\eins_{\{\widetilde\psi_k>a\}}\,d\mu_n-\int\widetilde\psi_k\eins_{\{\widetilde\psi_k>a\}}\,d\mu_0\Big|+C_{f,k}\int\widetilde\psi_k\eins_{\{\widetilde\psi_k>a\}}\,d\mu_0
\end{equation}
As see above, the second summand in (\ref{Charakterisierung psi-schwache Topologie - Proof - 40}) is at most $\varepsilon/5$. The first summand in (\ref{Charakterisierung psi-schwache Topologie - Proof - 40}) is bounded above by
\begin{equation}\label{Charakterisierung psi-schwache Topologie - Proof - 50}
    C_{f,k}\Big|\int\widetilde\psi_k\,d\mu_n-\int\widetilde\psi_k\,d\mu_0\Big|+C_{f,k}\Big|\int\widetilde\psi_k\eins_{\{\widetilde\psi_k\le a\}}\,d\mu_n-\int\widetilde\psi_k\eins_{\{\widetilde\psi_k\le a\}}\,d\mu_0\Big|.
\end{equation}
The first summand in (\ref{Charakterisierung psi-schwache Topologie - Proof - 50}) converges to $0$ as $n\to\infty$ by assumption. Thus we can find $n_0\in\N$ such that it is bounded above by $\varepsilon/5$ for every $n\ge n_0$. Since $\mu_0\circ \widetilde\psi_k^{-1}$ as a probability measure on the real line has at most countably many atom, we may and do assume that $a>0$ is chosen such that $\mu_0[\{\widetilde\psi_k=a\}]=0$. Since $\mu_n\to\mu_0$ weakly by assumption, it follows by the portmanteau theorem that the second summand in (\ref{Charakterisierung psi-schwache Topologie - Proof - 50}) converges to $0$ as $n\to\infty$. By possibly increasing $n_0$ we obtain that the second summand in (\ref{Charakterisierung psi-schwache Topologie - Proof - 50}) is at most $\varepsilon/5$ for all $n\ge n_0$. So far we have shown that the second summand in (\ref{Charakterisierung psi-schwache Topologie - Proof - 20}) is bounded above by $4\varepsilon/5$ for all $n\ge n_0$. Using the same arguments as for second summand in (\ref{Charakterisierung psi-schwache Topologie - Proof - 50}) and possibly increasing $n_0$ further, we moreover obtain that the first summand in (\ref{Charakterisierung psi-schwache Topologie - Proof - 20}) is bounded above by $\varepsilon/5$ for all $n\ge n_0$. That is, we indeed arrive at (\ref{Charakterisierung psi-schwache Topologie - Proof - 10}).\hfill\proofendsign

%%%%%%%%%%%%%%%%%%%%%%%%%%%%%%%%%%%%%%%%%%%%%%%%%%%%%%%%%%%%%%%%
%%%%%%%%%%%%%%%%%%%%%%%%%%%%%%%%%%%%%%%%%%%%%%%%%%%%%%%%%%%%%%%%
%%%%%%%%%%%%%%%%%%%%%%%%%%%%%%%%%%%%%%%%%%%%%%%%%%%%%%%%%%%%%%%%

\subsection{Proof of Theorem~\ref{Key Prop}}

First we shall provide the following characterization of relative compact subsets for the $(\psi_{k})$-weak topology, which will be needed in the proof of Lemma~\ref{compact lemma multivariate}. For the $\psi$-weak topology this characterization is already known from Corollary A.47 in \cite{FoellmerSchied2011}.

\begin{lemma}\label{psiweakcompactness}
Let $(\psi_{k})$ be any sequence of gauge functions and $M\subseteq{\cal M}_1^{(\psi_{k})}$ be given. Then the following conditions are equivalent:
\begin{enumerate}
    \item $M$ is relatively compact for the $(\psi_{k})$-weak topology.
    \item For every $k\in\bN$, $M$ is relatively compact for the $\psi_k$-weak topology.
    \item For every $k\in \bN$ and $\varepsilon>0$ there exists a compact set $K_k\subseteq E$ such that
    $$
        \sup_{\mu\in M}\int_{K_k^{\sf c}}\psi_k\,d\mu\le\eps.
    $$
    \item For every $k\in\bN$ there exists a measurable function $\phi_k:E\to\R_+$ such that each set $\{\phi_k\le n\psi_k\}$, $n\in\N$, is compact in $E$ and such that
    $$
        \sup_{\mu\in M}\int\phi_k\,d\mu<\infty.
    $$
\end{enumerate}
\end{lemma}

\begin{proof}
(b)$\Leftrightarrow$(c)$\Leftrightarrow$(d): These implications follow immediately from Corollary A.47 in \cite{FoellmerSchied2011}.

(a)$\Leftrightarrow$(b): Since the $(\psi_{k})$-weak and the $\psi_{k}$-weak topologies are metrizable, for any of these topologies the relatively compact subsets are exactly the relatively sequentially compact ones. Then the implication (a)$\Rightarrow$(b) is obvious. %{\color{magenta}Hier Abweichung vom Beweis von Lemma~\ref{Psi topology lemma 1}}
To prove the implication (b)$\Rightarrow$(a), let $M$ be relatively compact for the $\psi_k$-weak topology for each $k\in\bN$. In particular, every $\psi_{k}$-weak closure $M_{k}$ of $M$ in ${\cal M}_1^{\psi_{k}}$ is $\psi_{k}$-weakly compact. Then by Tychonoff's theorem the set $\times_{k=1}^\infty M_{k}$ is a compact subset of $\times_{k=1}^\infty{\cal M}_1^{\psi_{k}}$ for the product topology generated by the $\psi_{k}$-weak topologies. Notice that the product topology is metrizable by the metric
$$
    d_{{\rm prod}}(\boldsymbol{\mu},\boldsymbol{\nu}):= \sum_{k=1}^{\infty} 2^{-k} \big(d_{\psi_{k}}(\boldsymbol{\mu}(k),\boldsymbol{\nu}(k))\wedge 1\big).
$$
Let $(\mu_n)$ be any sequence in $M$. We will construct a $(\psi_k)$-weakly converging subsequence. For every $n\in\N$ we obtain an element $\boldsymbol{\mu}_n\in\times_{k=1}^\infty M_{k}$ by setting $\boldsymbol{\mu}_{n}(k):=\mu_n$, $k\in\N$. By compactness of $\times_{n=1}^\infty M_{k}$ we may extract a subsequence $(\boldsymbol{\mu}_{n(j)})$ from $(\boldsymbol{\mu}_{n})$ that converges to some $\boldsymbol{\mu}\in\times_{k=1}^\infty\cM_1^{\psi_k}$, i.e.,  $d_{{\rm prod}}(\boldsymbol{\mu}_{n(j)},\boldsymbol{\mu})\to 0$. In particular, $d_{\psi_{k}}(\boldsymbol{\mu}_{n(j)}(k),\boldsymbol{\mu}(k))\to 0$ for every $k\in\N$, i.e.,  $(\boldsymbol{\mu}_{n(j)}(k))$ converges $\psi_{k}$-weakly to $\boldsymbol{\mu}(k)$ for every $k\in\N$. Now, if we can show that $\boldsymbol{\mu}(k)=\boldsymbol{\mu}(1)=:\mu$ holds for every $k\in\N$, then it follows that $\mu_{n(j)}\to\mu$ $\psi_k$-weakly for every $k\in\N$ and thus $\mu_{n(j)}\to\mu$ $(\psi_k)$-weakly. In the rest of the proof we show that $\boldsymbol{\mu}(k)=\boldsymbol{\mu}(1)$ holds for every $k\in\N$.

For fixed $k\in\N$, the set $M$ is a subset of $\cM_1^{\psi_{1} + \psi_k}.$ Since $M$ is also a relatively $\psi_{i}$-compact subset of $\cM_1^{\psi_i}$ we may find by (c)
for every $\varepsilon > 0$ some compact subset $K_{i}\subseteq E$ such that
$$
    \sup_{\mu\in M}\int_{K_{i}^{\sf c}} \psi_{i}~d\mu\leq\varepsilon/2
$$
for $i\in\{1,k\}$. Then $K := K_{1}\cup K_{k}$ is a compact subset of $E$ such that
$$
    \sup_{\mu\in M}\int_{K^{\sf c}}(\psi_{1} + \psi_{k})\,d\mu \le \sup_{\mu\in M}\int_{K_{1}^{\sf c}} \psi_{1}\,d\mu + \sup_{\mu\in M}\int_{K_{k}^{\sf c}} \psi_{k}\,d\mu\le\varepsilon.
$$
Hence in view of Corollary A.47 in \cite{FoellmerSchied2011} the set $M$ is also a relatively compact subset of $\cM_1^{\psi_{1,k}}$ for the $\psi_{1,k}$-weak topology with $\psi_{1,k}:= \psi_{1} + \psi_k$. Therefore we may select a subsequence $(\mu_{n(j(\iota))})$ of $(\mu_{n(j)})$ which converges $\psi_{1,k}$-weakly to some $\mu\in \cM_1^{\psi_{1,k}}$. In particular, for every $f\in C_{\psi_{1}}\cup C_{\psi_{k}}$ we clearly have $f\in C_{\psi_{1,k}}$ and thus $\int f\,d\mu_{n(j(\iota))}\to\int f\,d\mu$. This means that $(\mu_{n(j(\iota))})$ converges to $\mu$ w.r.t.\ both the $\psi_{1}$-weak topology and the $\psi_{k}$-weak topology. This implies $\boldsymbol{\mu}(1) = \mu = \boldsymbol{\mu}(k)$, and the proof is complete.
\end{proof}

In the case where each set $\{\psi\le n\}$, $n\in\N$, is relatively compact in $E$, a set $M\subseteq{\cal M}_1^\psi$ is relatively compact for the $\psi$-weak topology if and only if it is uniformly $\psi$-integrating; cf.\ Lemma 3.4 in \cite{Zaehle2016}. The following Lemma shows that the same is true for general gauge functions (and for sequences $(\psi_{k})$ of general gauge functions) when $M$ is assumed to be relatively compact for the weak topology.

\begin{lemma}\label{compact lemma multivariate}
Let $(\psi_{k})$ be any sequence of gauge functions and $M\subseteq\cM_1$ be given. Then the following conditions are equivalent:
\begin{enumerate}
    \item $M$ is uniformly $(\psi_{k})$-integrating and relatively compact for the weak topology.
    \item $M$ is relatively compact for the $(\psi_{k})$-weak topology.
\end{enumerate}
\end{lemma}

\begin{proof}
(a)$\Rightarrow$(b): Let $k\in\N$ and $\varepsilon>0$ be given. Since $M$ is assumed to be uniformly $(\psi_k)$-integrating, there exists $a_{k}>0$ such that $\sup_{\mu\in M}\int\psi_{k}\eins_{\{\psi_{k}>a_{k}\}}\,d\mu\le \varepsilon/2$. Since $M$ is assumed to be weakly relatively compact, we moreover obtain by Prohorov's theorem a compact set $C_k\subseteq E$ such that $\sup_{\mu\in M}\mu[C_k^{\sf c}]\le\varepsilon/(2a_{k})$. The set $K_k:=C_k\cap\{\psi_{k}\le a_{k}\}$ is a compact subset of $E$ and satisfies $K_k^{\sf c}= \{\psi_k>a_{k}\}\cup(C_{k}^{\sf c}\cap\{\psi_k\le a_{k}\})$. Hence,
$$
    \sup_{\mu\in M}\int_{K_k^{\sf c}}\psi_{k}\,d\mu \,\le\, \sup_{\mu\in M}\int\psi_{k}\eins_{\{\psi_{k}>a_{k}\}}\,d\mu + \sup_{\mu\in M}\int_{C_k^{\sf c}}\psi_{k}\eins_{\{\psi_{k}\le a_{k}\}}\,d\mu \,\le\, \varepsilon.
$$
It follows by the implication (c)$\Rightarrow$(a) of Lemma~\ref{psiweakcompactness} that $M$ is relatively compact for the $(\psi_k)$-weak topology.

(b)$\Rightarrow$(a): By the implication (a)$\Rightarrow$(b) of Lemma~\ref{psiweakcompactness} the set $M$ is $\psi_{k}$-weakly relatively compact for each $k\in\N$. Hence $M$ is uniformly $\psi_{k}$-integrating for each $k\in\N$ due to Lemma A.2 in \cite{Kraetschmeretal2014}. Moreover, relative compactness of $M$ for the weak topology follows from the fact that the weak topology is coarser than the $(\psi_{k})$-weak topology.
\end{proof}

\bigskip

{\bf Proof of Theorem~\ref{Key Prop}:} (b)$\Rightarrow$(c): Let $M_0\subseteq M$ be weakly compact, and fix $\varepsilon>0$ and $k\in\N$. By assumption there exists for every $\mu\in M_0$ some weakly open neighborhood $U_\mu$ of $\mu$ and some $a_\mu>0$ such that $\int\psi_{k}\eins_{\{\psi_{k}\ge a_\mu\}}\,d\nu<\varepsilon$ for all $\nu\in U_\mu\cap M$. By weak compactness of $M_0$ we can extract a finite cover of $M_0$ consisting of such neighborhoods $U_{\mu_1},\ldots,U_{\mu_m}$ (with $\mu_1,\dots,\mu_m\in M$), and it follows that $\sup_{\nu\in M_0}\int\psi_k\eins_{\{\psi_k\ge a\}}\,d\nu\le\varepsilon$ if we take $a:=\max_{i=1,\ldots,m}a_{\mu_i}$.

(c)$\Rightarrow$(b): Let us suppose by way of contradiction that there exist $\mu\in M$, $k\in\N$, $\varepsilon>0$, and a sequence $(\nu_n)$ in $M$ such that $\nu_n\to\mu$ weakly but $\int\psi_{k}\eins_{\{\psi_{k}\ge n\}}\,d\nu_n\ge\varepsilon$ for all $n$. Then $\{\nu_1,\nu_2,\dots\}\cup\{\mu\}$ is weakly compact and not uniformly $(\psi_{k})$-integrating. This gives a contradiction.

(c)$\Rightarrow$(a):  Since both topologies are metrizable, it suffices to show that they coincide on any given weakly compact set $M_0\subseteq M$. By (c) and Lemma~\ref{compact lemma multivariate}, $M_0$ is compact for the $(\psi_{k})$-weak topology, and so the two topologies coincide on $M_0$ by Lemma \ref{relatively compact w set lemma}.

(a)$\Rightarrow$(c): Every weakly compact subset of $M$ is also $(\psi_{k})$-weakly compact due to (a), and hence uniformly $(\psi_{k})$-integrating by Lemma~\ref{compact lemma multivariate}.

(c)$\Leftrightarrow$(d): The implication (c)$\Rightarrow$(d) is obvious. Conversely suppose by way of contradiction that (d) holds but that there is a weakly compact $M_0\subseteq M$ that is not uniformly $(\psi_k)$-integrating. Then there exist $k\in\N$, $\varepsilon>0$, and a sequence $(\mu_n)$ in $M$ such that $\int\psi_k\eins_{\{\psi_k\ge n\}}\,d\mu_n\ge\varepsilon$ for all $n$. By selecting a weakly convergent subsequence we arrive at a contradiction to (d).

(a)$\Leftrightarrow$(e): This equivalence is obvious since both topologies are metrizable.
\hfill\proofendsign

%%%%%%%%%%%%%%%%%%%%%%%%%%%%%%%%%%%%%%%%%%%%%%%%%%%%%%%%%%%%%%%
%%%%%%%%%%%%%%%%%%%%%%%%%%%%%%%%%%%%%%%%%%%%%%%%%%%%%%%%%%%%%%%%
%%%%%%%%%%%%%%%%%%%%%%%%%%%%%%%%%%%%%%%%%%%%%%%%%%%%%%%%%%%%%%%%

\subsection{Proof of Theorem~\ref{Key Prop - Corollary}}

%%%%%%%%%%%%%%%%%%%%%%%%%%%%%%%%%%%%%%%%%%%%%%%%%%%%%%%%%%%%%%%
%%%%%%%%%%%%%%%%%%%%%%%%%%%%%%%%%%%%%%%%%%%%%%%%%%%%%%%%%%%%%%%%

\subsubsection{Proof of part (i)}

The proof of this part is organized as follows. Below we will show that conditions (a)--(b) of Lemma~\ref{hampel-huber generalized} and conditions (c)--(d) of Lemma~\ref{hampel-huber generalized - finite sample} are satisfied for every w-set $M\subseteq{\cal M}$ ($\subseteq{\cal M}_1^{(\psi_k)}$). Then, if for any w-set $M\subseteq\cM$ the functional $T$ is continuous at every $\mu\in M$ for the relative $(\psi_{k})$-weak topology ${\cal O}_{(\psi_{k})}\cap\cM$, the $(\psi_k)$-robustness of the sequence $(\widehat T_n)$ on $M$ is a consequence of the two lemmas and the fact that ${\cal O}_{\rm w}\cap M={\cal O}_{(\psi_k)}\cap M$ for every w-set $M$.

\begin{lemma}\label{hampel-huber generalized}
Let $M\subseteq{\cal M}$ and assume that the following two conditions hold:
\begin{enumerate}
    \item $T:{\cal M}\rightarrow\Sigma$ is $(d_{(\psi_k)},d_\Sigma)$-continuous at every $\mu\in M$.
    \item For every $\mu\in M$, $\varepsilon>0$, and $\eta>0$ there are some $\delta>0$ and $n_0\in\N$ such that
    $$
        \nu\in M,\quad d_{(\psi_k)}(\mu,\nu)\le\delta\quad\Longrightarrow\quad \pr^{\nu}\big[d_{(\psi_k)}(\widehat m_n,\nu)\ge\eta\big]\le\varepsilon \quad\mbox{for all }n\ge n_0.
    $$
\end{enumerate}
Then for every $\mu\in M$ and $\varepsilon>0$ there exist $n_0\in\N$ and an open neighborhood $U_{(\psi_k)}=U_{(\psi_k)}(\mu,\varepsilon;M)$ of $\mu$ for the relative $(\psi_k)$-weak topology ${\cal O}_{(\psi_k)}\cap M$ such that
$$
    \nu\in U_{(\psi_k)}\quad\Longrightarrow\quad \pi(\pr^{\mu}\circ\widehat T_n^{\,-1},\pr^{\nu}\circ\widehat T_n^{\,-1})\le\varepsilon \quad\mbox{for all }n\ge n_0.
$$
\end{lemma}

\begin{proof}
Note that the proof of Theorem 2.1 in \cite{Zaehle2016} still works when in assumption (a) of this theorem one only requires that the sequence $(V_n)$ is asymptotically $(d_\Upsilon,d_\Sigma)$-continuous at every point of $\Theta_0$ (and not on all of $\Theta$); take into account  that in the proof the asymptotic continuity of $(V_n)$ is used only subsequent to (41). Further note that in \cite{Zaehle2016} the assumption that the metric space $(\Upsilon,d_\Upsilon)$ be complete and separable is superfluous (and nowhere used). Then the claim follows by (the generalization of) Theorem 2.1 in \cite{Zaehle2016} with $(\Upsilon,d_\Upsilon):=({\cal M},d_{(\psi_k)})$, $U(\mu):=\mu$, $V_n:=T$ for all $n\in\N$, and $\widehat U_n(x_1,x_2,\ldots):=\widehat m_n(x_1,\ldots,x_n)$.
\end{proof}

For every $n\in\N$ we equip the $n$-fold product space $E^n$ with the product topology. Note that the corresponding Borel $\sigma$-field coincides with the $n$-fold product ${\cal B}(E)^{\otimes n}$ of the Borel $\sigma$-algebra ${\cal B}(E)$ on $E$, and let $\pi_n$ be any metric that metrizes the weak topology on the set of all probability measures on $(E^n,{\cal B}(E)^{\otimes n})$. Let $X_i$ be the $i$-th coordinate projection on $\Omega=E^\N$ and let $\widehat t_n:E^n\rightarrow\Sigma$ be the estimator $\widehat T_n$ regarded as a map on $E^n$; recall (\ref{Def Plug-in Estimator}) and note that $\widehat m_n(x)$ depends only on the first $n$ coordinates of $x=(x_1,x_2,\ldots)\in E^\N$.

\begin{lemma}\label{hampel-huber generalized - finite sample}
Let $M\subseteq{\cal M}$ and assume that the following two conditions hold:
\begin{itemize}
    \item[{\rm (c)}] $E^n\ni(x_1,\ldots,x_n)\mapsto\widehat t_n(x_1,\ldots,x_n)$ is continuous for every $n\in\N$.
    \item[{\rm (d)}] $M\ni\mu\mapsto\pr^\mu\circ(X_1,\ldots,X_n)^{-1}$ is $(d_{\rm w},\pi_n)$-continuous for every $n\in\N$.
\end{itemize}
Then for every $\mu\in M$, $n\in\N$, and $\varepsilon>0$ there exist and an open neighborhood $U_{(\psi_k)}=U_{(\psi_k)}(\mu,\varepsilon;M)$ of $\mu$ for the relative $(\psi_k)$-weak topology ${\cal O}_{(\psi_k)}\cap M$ such that
$$
    \nu\in U_{(\psi_k)}\quad\Longrightarrow\quad \pi(\pr^{\mu}\circ\widehat T_n^{\,-1},\pr^{\nu}\circ\widehat T_n^{\,-1})\le\varepsilon.
$$
\end{lemma}

\begin{proof}
The lemma is a direct consequence of Theorem 2.5 and Example 2.6 in  \cite{Zaehle2016}.
\end{proof}

As already discussed at the beginning of the proof, it remains to show that conditions (a)--(d) are satisfied, where for (b) we have to assume that $M\subseteq{\cal M}$ is a w-set in ${\cal M}_1^{(\psi_k)}$.

(a): Condition (a) holds by assumption.

(b): To verify condition (b) for any fixed w-set $M\subseteq{\cal M}$ ($\subseteq{\cal M}_1^{(\psi_k)}$), we assume without loss of generality that the metric $d_{\rm w}$ in (\ref{Det psi k metric}) is given by the Prohorov metric $d_{\mbox{\scriptsize{\rm P}}}$, i.e.,
$$
    d_{(\psi_k)}(\mu,\nu)=d_{\mbox{\scriptsize{\rm P}}}(\mu,\nu)+\sum_{k=1}^\infty 2^{-k}\Big(\Big|\int\psi_k\,d\mu_1-\int\psi_k\,d\mu_2\Big|\wedge 1\Big).
$$
Let $\mu\in M$, $\varepsilon>0$, and $\eta>0$ be fixed. Choose $k_{\eta}\in\N$ so large such that $\sum_{k=k_\varepsilon+1}^\infty2^{-k}<\eta/3$. Then, for every $\nu\in{\cal M}$,
\begin{eqnarray}
    \lefteqn{\pr^{\nu}\big[d_{(\psi_k)}(\widehat m_n,\nu)\ge\eta\big]}\nonumber\\
    & \le & \pr^{\nu}\big[d_{(\psi_k)}(\widehat m_n,\nu)\ge\eta/3\big]+\sum_{k=1}^{k_\varepsilon}\pr^{\nu}\Big[2^{-k}\Big|\int\psi_k\,d\widehat m_n-\int\psi_k\,d\nu\Big|\ge\eta/(3k_\varepsilon)\Big].\nonumber
\end{eqnarray}
Lemma 4 in \cite{Mizera2010} shows that $\lim_{n\to\infty}\sup_{\nu\in{\cal M}_1}\pr^\nu[d_{\mbox{\scriptsize{\rm P}}}(\widehat m_n,\nu)\ge\eta]=0$. So we can find some $n_{\mbox{\scriptsize{\rm P}}}\in\N$ such that
\begin{equation}
    \sup_{\nu\in{\cal M}_1}\,\pr^\mu[d_{\mbox{\scriptsize{\rm P}}}(\widehat m_n,\nu)\ge\eta]\,\le\,\varepsilon/2\quad\mbox{ for all }n\ge n_{\mbox{\scriptsize{\rm P}}}.
\end{equation}
So it remains to show that for every $k=1,\ldots,k_\varepsilon$ there exist $\delta_k>0$ and $n_k\in\N$ such that
\begin{equation}\label{hampel-huber generalized - plug-in estimator - proof - 10}
    \nu\in M,~~ d_{(\psi_k)}(\mu,\nu)\le\delta_k~~\Longrightarrow~~\pr^{\nu}\Big[\Big|\int\psi_k\,d\widehat m_n-\int\psi_k\,d\nu\Big|\ge\frac{2^k\eta}{3k_\varepsilon}\Big]\le\frac{\varepsilon}{2k_\varepsilon}\mbox{ for all }n\ge n_k.
\end{equation}
By choosing $\delta:=\min\{\delta_1,\ldots,\delta_{k_\varepsilon}\}$ and $n_0:=\max\{n_{\mbox{\scriptsize{\rm P}}},n_1,\ldots,n_{k_\varepsilon}\}$ we then obtain (b).

To prove (\ref{hampel-huber generalized - plug-in estimator - proof - 10}), we take into account that $M$ is a w-set in ${\cal M}_1^{(\psi_k)}$. By Theorem~\ref{Key Prop} this means that $M$ is locally uniformly $(\psi_k)$-integrating. Thus for every $k\in\N$ we can find some $\delta_k>0$ and $a_k>0$ such that $\int\psi_k\eins_{\{\psi_k\ge a_k\}} d\nu<\min\{\frac{2^k\eta}{3k_\varepsilon};\frac{2^{k}\eta}{9k_\varepsilon}\frac{\varepsilon}{2}\}$ for all $\nu\in M$ with $d_{\mbox{\scriptsize{\rm P}}}(\mu_1,\mu_2)\le\delta$. For every $\nu\in M$ with $d_{\mbox{\scriptsize{\rm P}}}(\mu,\nu)\le\delta_k$ we then obtain
\begin{eqnarray}
    \lefteqn{\pr^{\nu}\Big[\Big|\int\psi_k\,d\widehat m_n-\int\psi_k\,d\nu\Big|\ge\frac{2^k\eta}{3k_\varepsilon}\Big]}\nonumber\\
    & \le & \pr^{\nu}\Big[\int\psi_k\eins_{\{\psi_k\ge a_k\}}\,d\widehat m_n\ge\frac{2^k\eta}{9k_\varepsilon}\Big]\nonumber\\
    & & +\,\pr^{\nu}\Big[\,\Big|\int\psi_k\eins_{\{\psi_k<a_k\}}\,d\widehat m_n-\int\psi\eins_{\{\psi<a_k\}}\,d\nu\Big|\ge\frac{2^k\eta}{9k_\varepsilon}\Big]\nonumber\\
    & & +\,\pr^{\nu}\Big[\int\psi_k\eins_{\{\psi_k\ge a_k\}}\,d\nu\ge\frac{2^k\eta}{9k_\varepsilon}\Big]\nonumber\\
    & =: & S_1(k,n,a_k)+S_2(k,n,a_k)+S_3(k,a_k),\nonumber
\end{eqnarray}
where $S_3(k,a_k)=0$ and $S_1(k,n,a_k)\le(9k_\varepsilon/(2^k\eta)\int \psi\eins_{\{\psi_k\ge a_k\}}d\nu\le\varepsilon/2$ for all $n\in\N$ (by Markov's inequality). Further, by Chebychev's inequality we can find some $n_k\in\N$ such that $S_2(k,n,a)\le\varepsilon/2$ for all $n\ge n_k$ (and all $\nu\in{\cal M}_1$). This proves (\ref{hampel-huber generalized - plug-in estimator - proof - 10}) with $d_{(\psi_k)}$ replaced by $d_{\mbox{\scriptsize{\rm P}}}$. Since $d_{\mbox{\scriptsize{\rm P}}}\le d_{(\psi_k)}$, we arrive at (\ref{hampel-huber generalized - plug-in estimator - proof - 10}).

(c): The mapping $(x_1,\ldots,x_n)\mapsto \widehat t_n(x_1,\ldots,x_n)=T(\frac{1}{n}\sum_{i=1}^n\delta_{x_i})$ is continuous, because the statistical functional $T$ is $(d_{(\psi_k)},d_\Sigma)$-continuous by assumption and the mapping $(x_1,\ldots,x_n)\mapsto\frac{1}{n}\sum_{i=1}^n\delta_{x_i}$ is easily seen to be $(d_{E^n},d_{(\psi_k)})$-continuous, where $d_{E_n}$ is any metric which metrizes the product topology on $E^n$.

(d): The $(d_{\rm w},\pi_{n})$-continuity of the mapping $M\ni\mu\mapsto\pr^\mu\circ(X_1,\ldots,X_n)^{-1}=\mu^{\otimes n}$ for every $n\in\N$ is obvious. too.
\hfill\proofendsign

%%%%%%%%%%%%%%%%%%%%%%%%%%%%%%%%%%%%%%%%%%%%%%%%%%%%%%%%%%%%%%%
%%%%%%%%%%%%%%%%%%%%%%%%%%%%%%%%%%%%%%%%%%%%%%%%%%%%%%%%%%%%%%%%

\subsubsection{Proof of part (ii)}

Now assume that $(\widehat T_n)$ is $(\psi_k)$-robust and weakly consistent. The $(\psi_k)$-robustness means that $(\widehat T_n)$ is robust on every w-set $M\subseteq{\cal M}$ ($\subseteq{\cal M}_1^{(\psi_k)}$). By the classical Hampel theorem (Theorem 1 in \cite{Cuevas1988}) we can conclude that $T|_{M}$ is weakly continuous for every w-set $M\subseteq{\cal M}$ ($\subseteq{\cal M}_1^{(\psi_k)}$). In the remainder we will show that this implies $(\psi_k)$-weak continuity of $T$. Let $\mu,\mu_1,\mu_2,\ldots\in \mathcal{M}$  such that $\mu_n\to\mu$ $(\psi_k)$-weakly.
%\red{\sout{ (in particular, $\mu_n\to\mu$ weakly). We have to show that $d_\Sigma(T(\mu_n),T(\mu))\to 0$.}}
Since $T|_{M}$ is $(d_{\rm w},d_\Sigma)$-continuous for every w-set $M\subseteq{\cal M}$ ($\subseteq{\cal M}_1^{(\psi_k)}$), it suffice to show that the set $M:=\{\mu,\mu_1,\mu_2,\ldots\}$ is a w-set in $M\subseteq{\cal M}$ ($\subseteq{\cal M}_1^{(\psi_k)}$). By assumption, the set $M$ is compact for the $(\psi_{k})$-weak topology since this topology is metrizable. Thus by Lemma \ref{relatively compact w set lemma} the set $M$ is also a w-set in $M\subseteq{\cal M}$. This completes the proof.
\hfill\proofendsign

%%%%%%%%%%%%%%%%%%%%%%%%%%%%%%%%%%%%%%%%%%%%%%%%%%%%%%%%%%%%%%%%
%%%%%%%%%%%%%%%%%%%%%%%%%%%%%%%%%%%%%%%%%%%%%%%%%%%%%%%%%%%%%%%%
%%%%%%%%%%%%%%%%%%%%%%%%%%%%%%%%%%%%%%%%%%%%%%%%%%%%%%%%%%%%%%%%
%%%%%%%%%%%%%%%%%%%%%%%%%%%%%%%%%%%%%%%%%%%%%%%%%%%%%%%%%%%%%%%%
%%%%%%%%%%%%%%%%%%%%%%%%%%%%%%%%%%%%%%%%%%%%%%%%%%%%%%%%%%%%%%%%
%%%%%%%%%%%%%%%%%%%%%%%%%%%%%%%%%%%%%%%%%%%%%%%%%%%%%%%%%%%%%%%%

\section{Remaining proofs}\label{Sec Remaining proofs}

%%%%%%%%%%%%%%%%%%%%%%%%%%%%%%%%%%%%%%%%%%%%%%%%%%%%%%%%%%%%%%%%
%%%%%%%%%%%%%%%%%%%%%%%%%%%%%%%%%%%%%%%%%%%%%%%%%%%%%%%%%%%%%%%%
%%%%%%%%%%%%%%%%%%%%%%%%%%%%%%%%%%%%%%%%%%%%%%%%%%%%%%%%%%%%%%%%

\subsection{Proof of Proposition~\ref{Aggregationsrobustheit}}\label{BeweisAggregationsrobustheit}

For every $i=1,\ldots,d$ we define $\mu_i':=\mu_i\circ f_{d,c}^{-1}$ for $f_{d,c}(x):=(d+1)cx$. By assumption (b) $\{\mu_1',\ldots,\mu_d'\}$ is a finite subset of
${\cal M}_{1}^{(\psi_{k})}$, and thus uniformly $(\psi_k)$-integrating. In view of de la Vall\'ee-Poussin theorem for sets of measures (analogue Theorem II.T22 in \cite{Meyer1966}) one can thus find for every $k\in\N$ a convex and increasing function $h_k:\R_+\rightarrow\R_+$ such that $\lim_{x\to\infty}h_k(x)/x=\infty$ and
\begin{equation}\label{BeweisAggregationsrobustheit - 10}
    \max_{i=1,\ldots,d}\int h_k(\psi_k)\,d\mu_i'\,<\,\infty.
\end{equation}
Since $\psi_k$ is convex and nonnegative, it is also nondecreasing on $[0,\infty)$. In addition $\psi_{k}$ is assumed to be even, so that the composition $h_k\circ\psi_k = h_{k}\circ\psi_{k}(|\cdot|)$ is convex. Together with assumption (a) and (\ref{BeweisAggregationsrobustheit - 10}) this yields
\begin{eqnarray*}
    \int h_k(\widetilde\psi_k)\,d\boldsymbol{\mu}\circ A_d^{-1}
    & = & \int h_k\circ\widetilde\psi_k(A_d(\boldsymbol{x}))\,\boldsymbol{\mu}(d\boldsymbol{x})\\
    & = & \int h_k\circ\psi_k(\|A_d(\boldsymbol{x})\|)\,\boldsymbol{\mu}(d\boldsymbol{x})\\
    & \le & \int h_k\circ\psi_k\Big(b+c\sum_{i=1}^d|x_i|\Big)\,\boldsymbol{\mu}(d\boldsymbol{x})\\
    & \le & \int\frac{1}{d+1}\sum_{i=0}^d h_k\circ\psi_k\big((d+1)cx_i\big)\,\boldsymbol{\mu}(d\boldsymbol{x})\\
    & \le & h_k\circ\psi_k\big((d+1)b\big)\,\vee\,\max_{j=1,\ldots,d}\int h_k\circ\psi_k\big((d+1)cx\big)\,\mu_j(dx)\\
    & = & h_k\circ\psi_k\big((d+1)b\big)\,\vee\,\max_{j=1,\ldots,d}\int h_k\circ\psi_k(x)\,\mu_j'(dx)~<~\infty
\end{eqnarray*}
for all $k\in\N$ and $\boldsymbol{\mu}\in \boldsymbol{{\cal M}}(d;N)$, where we used the convention $x_0:=b/c$. This implies
$$
    \sup_{\nu\in M}\int h_k(\widetilde\psi_k)\,d\nu\,<\,\infty
$$
for all $k\in\N$, and by another application of the de la Vall\'ee-Poussin theorem for sets of measures we can conclude that $M$ is uniformly $(\widetilde\psi_k)$-integrable.
\hfill\proofendsign

%%%%%%%%%%%%%%%%%%%%%%%%%%%%%%%%%%%%%%%%%%%%%%%%%%%%%%%%%%%%%%%%
%%%%%%%%%%%%%%%%%%%%%%%%%%%%%%%%%%%%%%%%%%%%%%%%%%%%%%%%%%%%%%%%
%%%%%%%%%%%%%%%%%%%%%%%%%%%%%%%%%%%%%%%%%%%%%%%%%%%%%%%%%%%%%%%%

\subsection{Proof of Theorem~\ref{contrastfunction}}

Let $(\mu_{n})_{n\in\N}$ be any sequence in $\cM$ such that $\cL_{\mu_{n}}(\theta_{k})\to \cL_{\mu}(\theta_{k})$ holds for every $k\in\N$. In particular the sequence $(\cL_{\mu_{n}})_{n\in\N}$ converges pointwise on a dense subset of $\Theta$ to $\cL_{\mu}$. Together with the concavity of $\cL_{\mu}$ and $\cL_{\mu_{n}}$, $n\in\N$, this implies that $(\cL_{\mu_{n}})_{n\in\N}$ converges even pointwise to $\cL_{\mu}$; cf.\ Corollary 7.18 in \cite{RockafellarWets1998}.

Further, by assumption $\argmax_{\theta\in\Theta}\cL_{\mu}(\theta) = \{T(\mu)\}$ and $T(\mu_{n})\in \argmax_{\theta\in\Theta}\,\cL_{\mu_{n}}(\theta)$ for every $n\in\N$. Since $(\cL_{\mu_{n}})_{n\in\N}$ is a sequence of concave maps which converges pointwise to the concave map $\cL_{\mu}$, we may draw on well-known results concerning stability of convex minimization (e.\,g.\ Theorem 5.3.25(f) in \cite{KosmolMueller-Wichards2011}) to conclude $T(\mu_{n})\to T(\mu)$. So the first part of Theorem~\ref{contrastfunction} is shown. The remaining part follows immediately from the first part, because convergence $\mu_{n}\to\mu$ w.r.t.\ the $(\psi_{k})$-weak topology implies $\cL_{\mu_{n}}(\theta_{k})\to \cL_{\mu}(\theta_{k})$ for every $k\in\N$. Now, the proof is complete.
\hfill\proofendsign

%%%%%%%%%%%%%%%%%%%%%%%%%%%%%%%%%%%%%%%%%%%%%%%%%%%%%%%%%%%%%%%%
%%%%%%%%%%%%%%%%%%%%%%%%%%%%%%%%%%%%%%%%%%%%%%%%%%%%%%%%%%%%%%%%
%%%%%%%%%%%%%%%%%%%%%%%%%%%%%%%%%%%%%%%%%%%%%%%%%%%%%%%%%%%%%%%%

\subsection{Proof of Theorem~\ref{StetigkeitRisikofunktional}}

It is known from Theorem 2.1.11 in \cite{EdgarSucheston1992} that the Orlicz heart $H^\Psi$ is a Banach space when endowed with the Luxemburg norm
$$
    \|X\|_{\Psi} := \inf\big\{\lambda > 0:\,\bE [\Psi(|X|/\lambda)]\leq 1\big\}.
$$
Moreover, we may observe that $\|X\|_{\Psi}\leq\|X\|_{\Psi}$ whenever $|X|\leq |Y|$ $\pr$-a.s. This means that $H^{\Psi}$ equipped with $\|\cdot\|_{\Psi}$ and the $\pr$-a.s.\ order is a Banach lattice. It follows by Proposition 3.1 in \cite{RuszczynskiShapiro2006} that
\begin{equation}\label{StetigkeitRisikomass}
    \rho\mbox{ is continuous w.r.t.}~\|\cdot\|_{\Psi}.
\end{equation}
The missing link between (\ref{StetigkeitRisikomass}) and Theorem~\ref{StetigkeitRisikofunktional} is provided by the following representation result which is interesting in its own right. Recall that $\OFP$ is atomless so that it supports a random variable which is uniformly distributed on the open unit interval.

\begin{theorem}\label{SDWn}
A sequence $(\mu_n)$ in $\cM_{1}^{(\psi_{k})}$ converges w.r.t.\ the $(\psi_{k})$-weak topology to some $\mu_0\in\cM_{1}^{(\psi_{k})}$ if and only if $ \|F^{\leftarrow}_{\mu_n}(U) - F^{\leftarrow}_{\mu_0}(U)\|_\Psi\to 0$, where $U$ is an arbitrary random variable on $(\Omega,\cF,\pr)$ that is uniformly distributed on (0,1).
\end{theorem}

\begin{proof}
We let $X_n:=F^{\leftarrow}_{\mu_n}(U)$ and  prove first that $\|X_n -X_0\|_\Psi\to 0$ implies that $\mu_n\to\mu_0$ in the $(\psi_{k})$-weak topology. By Proposition 2.1.10 in \cite{EdgarSucheston1992}, $\|X_n -X_0\|_\Psi\to 0$ yields $\bE[\psi_{2k}(X_n-X_0)]\to 0$ for all $k\in\bN$ and $X_n\to X_0$ in probability. Convexity and monotonicity of $\Psi$ imply that $0\le\psi_k(X_n))\le \frac12 \psi_{2k}(X_n-X_0)+\frac12\psi_{2k}(X_0)$. Hence, $\psi_k(X_n)$ is uniformly integrable, and we obtain by Vitali's theorem in the form of \cite[Proposition 3.12\,(iii)$\Rightarrow$(ii)]{Kallenberg1997} that
$$
    \int\psi_k(x)\,\mu_n(dx)=\bE[\psi_k(X_n)]\longrightarrow\bE[\psi_k(X_0)]=\int\psi_k(x)\,\mu_0(dx).
$$
Moreover, since $X_n\to X_0$ $\pr$-a.s., the corresponding laws $(\mu_n)$ converge weakly. It follows that $(\mu_n)$ converges to $\mu$ w.r.t.~the $(\psi_{k})$-weak topology.

Conversely, assume that $\mu_n\to\mu_0$ in the $(\psi_{k})$-weak topology. Then $\mu_n\to\mu_0$ weakly, and the continuity of $\Psi$ and the fact that $\Psi(0)=0$ yield that
\begin{align}
    \psi_k(X_n) \longrightarrow \psi_k(X_0) & \qquad \bP\mbox{-a.s. for all $k\ge0$,} \label{SDW - proof - eq - 1} \\
    \psi_k(X_n-X_0) \longrightarrow 0 & \qquad \bP\mbox{-a.s. for all $k\ge0$.}\label{SDW - proof - eq - 2}
\end{align}
Moreover, the  convergence   $\mu_n\to\mu_0$ w.r.t.~the $(\psi_{k})$-weak topology implies that
\begin{equation}\label{SDW - proof - eq - 3}
    \bE[\psi_k(X_n)]=\int\psi_k(x)\,\mu_n(dx)\longrightarrow \int\psi_k(x)\,\mu_0(dx)=\bE[\psi_k(X_0)].
\end{equation}
In particular each expectation $\bE[\psi_k(X_n)]$ is finite so that we have $X_n\in H^\Psi$. Now, \eqref{SDW - proof - eq - 1}, (\ref{SDW - proof - eq - 3}), and Vitali's theorem in the form of \cite[Proposition 3.12\,(ii)$\Rightarrow$(iii)]{Kallenberg1997} imply that the sequence $(\psi_k(X_n))_{n\in\N_{0}}$ is uniformly  integrable for every $k$. Since $\Psi$ is nondecreasing and convex  we obtain $\psi_k(X_{n}- X_{0})\le\frac{1}{2}\psi_{2k}( X_{0}) + \frac12\psi_{2k}( X_{0})$. Since the sequence $(\psi_{2k}(X_n))_{n\in\N_{0}}$ is uniformly integrable, we may thus conclude that the sequence $(\psi_{k}(X_n-X_0))_{n\in\N}$ is uniformly integrable. Therefore, \eqref{SDW - proof - eq - 2} and another application of Vitali's theorem, this time in the form of \cite[Proposition 3.12\,(iii)$\Rightarrow$(ii)]{Kallenberg1997}, yield $\bE[\psi_k(X_n-X_0)]\to 0$ for every $k>0$, which implies $\|X_n-X_0\|_\Psi\to 0$  according to Proposition 2.1.10 in \cite{EdgarSucheston1992}.
\end{proof}

\bigskip

{\bf Proof of Theorem~\ref{StetigkeitRisikofunktional}:} Since $\pr\circ(F^{\leftarrow}_{\nu}(U))^{-1}=\nu$ for every $\nu\in{\cal M}_1$, the asserted $(\psi_{k})$-weak continuity of the risk functional $\cR_{\rho}$ is an immediate consequence of (\ref{StetigkeitRisikomass}) and Theorem~\ref{SDWn}.
\hfill\proofendsign

%%%%%%%%%%%%%%%%%%%%%%%%%%%%%%%%%%%%%%%%%%%%%%%%%%%%%%%%%%%%%%%%
%%%%%%%%%%%%%%%%%%%%%%%%%%%%%%%%%%%%%%%%%%%%%%%%%%%%%%%%%%%%%%%%
%%%%%%%%%%%%%%%%%%%%%%%%%%%%%%%%%%%%%%%%%%%%%%%%%%%%%%%%%%%%%%%%
%%%%%%%%%%%%%%%%%%%%%%%%%%%%%%%%%%%%%%%%%%%%%%%%%%%%%%%%%%%%%%%%
%%%%%%%%%%%%%%%%%%%%%%%%%%%%%%%%%%%%%%%%%%%%%%%%%%%%%%%%%%%%%%%%
%%%%%%%%%%%%%%%%%%%%%%%%%%%%%%%%%%%%%%%%%%%%%%%%%%%%%%%%%%%%%%%%
%%%%%%%%%%%%%%%%%%%%%%%%%%%%%%%%%%%%%%%%%%%%%%%%%%%%%%%%%%%%%%%%
%%%%%%%%%%%%%%%%%%%%%%%%%%%%%%%%%%%%%%%%%%%%%%%%%%%%%%%%%%%%%%%%
%%%%%%%%%%%%%%%%%%%%%%%%%%%%%%%%%%%%%%%%%%%%%%%%%%%%%%%%%%%%%%%%

%%%%%%%%%%%%%%%%%%%%%%%%%%%%%%%%%%%%%%%%%%%%%%%%%%%%%%%%%%%%%%%%
%%%%%%%%%%%%%%%%%%%%%%%%%%%%%%%%%%%%%%%%%%%%%%%%%%%%%%%%%%%%%%%%
%%%%%%%%%%%%%%%%%%%%%%%%%%%%%%%%%%%%%%%%%%%%%%%%%%%%%%%%%%%%%%%%

\end{document}